\documentclass[11pt]{aims}
\usepackage{amsmath, amssymb, mathrsfs, bbm}
\usepackage[pagewise]{lineno}
  \usepackage{paralist}
  \usepackage{graphics} 
  \usepackage{epsfig} 
\usepackage{graphicx} 
 \usepackage{epstopdf}
 \usepackage[colorlinks=true]{hyperref}
 \hypersetup{urlcolor=blue, citecolor=blue}
 \usepackage[toc,page]{appendix}
\usepackage{chngcntr}
\usepackage{hyperref}
\usepackage{url}

\usepackage{titlesec}
\titleformat{\section}{\Large\bfseries}{\thesection.}{4pt}{}
\titleformat{\subsection}{\large\bfseries}{\thesection.\arabic{subsection}.}{4pt}{}
\titleformat{\subsubsection}{\bfseries}{\thesection.\arabic{subsection}.\arabic{subsubsection}.}{4pt}{}
\titleformat*{\paragraph}{\bfseries}
\titleformat*{\subparagraph}{\bfseries}
\usepackage[margin=1in]{geometry}

  \def\currentyear{}
   \def\currentmonth{}
    \def\ppages{}
     \def\DOI{}

\newtheorem{theorem}{Theorem}[section]

\newtheorem{lemma}[theorem]{Lemma}
\newtheorem{proposition}[theorem]{Proposition}
\theoremstyle{definition}
\newtheorem{definition}[theorem]{Definition}
\newtheorem{remark}[theorem]{Remark}

\newcommand{\RN}{\mathbb{R}^N}
\newcommand{\R}{\mathbb{R}}
\newcommand{\N}{\mathbb{N}}

\numberwithin{equation}{section}

\title[Blowup profile for the gradient]{Sharp equivalent for the blowup  profile  to the gradient of a solution to  the semilinear  heat equation }

\author[G. K. Duong,   T. E. Ghoul and H. Zaag ]{}
 \keywords{Blowup solution, Blowup profile, Stability, Semilinear heat equation, non variation heat equation.\\
 \textbf{Acknowledgement:} The author Giao Ky Duong is  supported by The International Center for Research and Postgraduate Training in Mathematics-Institute of Mathematics Vietnam Academy of Science and
Technology under the Grant ICRTM04 2021.05, and partially by  University of Economics Ho Chi Minh City.  \\
\textbf{Data availability}: This paper has no associated data.}

\thanks{\today}

\begin{document}
\maketitle

\centerline{Giao Ky Duong$^{(1)}$,   Tej Eddine Ghoul$^{(2)}$   and Hatem Zaag$^{(3)}$} 
\medskip
{\footnotesize
\centerline{$^{(1)}$ Institute of Applied Mathematcis, University of Economics  Ho Chi Minh City, Vietnam.}
  \centerline{ $^{(2)}$ NYUAD Research Institute, New York University Abu Dhabi, PO Box 129188, Abu Dhabi, UAE}
   \centerline{ $^{(3)}$ Universit\'e Sorbonne Paris Nord,
LAGA, CNRS (UMR 7539), F-93430, Villetaneuse, France.}
}

\bigskip
\begin{center}\thanks{\today}\end{center}

\begin{abstract} 
 In this paper,   we  consider    the   standard semilinear  heat equation
\begin{eqnarray*}
\partial_t u = \Delta u  + |u|^{p-1}u, \quad p >1. 
\end{eqnarray*} 
\noindent
 The determination of the (believed to be) generic blowup profile is well-established in the literature, with the solution blowing up only at one point. Though the blow-up of the gradient of the solution is a direct consequence of the single-point blow-up property and the mean value theorem, there is no determination of the final blowup profile for the gradient in the literature, up to our knowledge.
In this  paper, we refine the construction technique of Bricmont-Kupiainen \cite{BKnon94} and Merle-Zaag \cite{MZdm97}, and derive the following profile for the gradient:
$$  \nabla u(x,T)   \sim     -  \frac{\sqrt{2b}}{p-1} \frac{x}{|x| \sqrt{ |\ln|x||}}    \left[\frac{b|x|^2}{2|\ln|x||} \right]^{-\frac{p+1}{2(p-1)}}  \text{ as   } x \to 0,    $$
where $ b =\frac{(p-1)^2}{4p}$, which is  \textbf{\textit{as expected}} the gradient of the well-known blowup profile of the solution.
\end{abstract}

\maketitle
\section{Introduction}


In this paper,  we consider    the following semilinear heat equation
\begin{equation}\label{equa-u-vortex}
\left\{\begin{array}{rcl}
\partial_t  u  & = & \Delta u + |u|^{p- 1}u\\[0.2cm]
u(0) & =& u_0 \in L^\infty(\R^N)
\end{array}
\right.,
\end{equation}
where $u : (x,t) \in \mathbb{R}^N \times [0,T)  \to  \R$ and $p>1$.  The local in time Cauchy problem can be solved in $L^\infty$, thanks to a fixed point technique (see  
Quittner and  Souplet \cite{QPbook07} for instance).  Roughly speaking,   for each initial data  $u_0 \in L^\infty$,  one of the following cases holds:
\begin{itemize}
    \item Either  the  solution  is global.
    \item Or it blows up in finite time $T$  i.e 
    $$ \limsup_{t \to T} \|u(t)\|_{L^\infty} \to +\infty.$$
\end{itemize}
In this case, $T$ is called the blowup time, and if for some 
$a \in \R^N$, there exists $(a_n,t_n) \to (a,T)$  as $n \to +\infty$  such that
$$  |u(a_n,t_n)| \to +\infty \text{ as } n \to +\infty,  $$
then $a$ is called a blowup point.
\bigskip

The behavior of solutions at blowup has generated a huge literature.
The book by Quittner and  Souplet \cite{QPbook07} is a good resource on the subject. According to Herrero and Vel\'azquez, the generic blowup behavior corresponds to the situation where the solution blows up only at one blowup point (say, the origin, from invariance by translation in space of equation \eqref{equa-u-vortex})
with the final blow-up profile 
\begin{equation}\label{defprofile}
u(x,T) \equiv \lim_{t\to T}u(x,t) 
\end{equation}
satisfying 
\begin{equation}\label{equiv}
u(x,T) \sim u^*(x) \equiv \left[ \frac{(p-1)^2}{8p}  \frac{|x|^2}{|\ln|x||} \right]^{-\frac{1}{p-1}}
\mbox{ as }x\to 0.
\end{equation}
Note that the genericity of this profile was published only in one space dimension in Herrero and Vel\'azquez \cite{HVasnsp92} (see also the note \cite{HVcras94}). The higher dimensional case was also proved by the same authors, but never published.
%
%
Several papers proved the existence of solutions obeying this generic behavior (see Galaktionov et al \cite{VGHZVMMF91} for a formal approach, Berger and Kohn \cite{BKcpam88} for a formal and a numerical evidence, Herrero and Vel\'azaquez \cite{HVaihn93} for a rigorous proof, and also Bricmont and Kupiainen \cite{BKnon94} together with Merle and Zaag \cite{MZdm97}).

\medskip

Since the origin is the unique blowup point in the considered behavior, we easily see from the mean value theorem that $\nabla u$ blows up at the origin as well. Accordingly, 
as the convergence in \eqref{defprofile} holds uniformly on every compact set of $\mathbb{R}^N\backslash \{0\}$, one may wonder whether the estimate \eqref{equiv} holds after differentiation in space. Up to our knowledge, such a result is not available in the literature, and we only have an upper bound proved by Abdelhedi and Zaag in \cite{AZarix20}:
\[
|\nabla u(x,T)|\le C|x|^{-\frac{p+1}{p-1}}|\ln |x||^{\frac{p+3}{4(p-1)}}.
\]
In this paper, we sharply adapt the construction method of \cite{BKnon94} and \cite{MZdm97} and show that \eqref{equiv} holds after differentiation in space.

\bigskip

That method was first introduced by Bressan in \cite{Brejde92} and \cite{Breiumj90} for the heat equation with an exponential source, then in Bricmont and Kupiainen \cite{BKnon94} and Merle and Zaag \cite{MZdm97} for equation \eqref{equa-u-vortex}. It consists in a formal approach where the profile is obtained through an inner/outer expansion with matching asymptotics, followed by a rigorous proof where the PDE is linearized around the profile candidate. Then, the negative part of the spectrum (which is infinite dimensional) is controlled thanks to the decaying properties of the Laplacian, whereas the nonnegative part (which is finite dimensional) is controlled thanks to the degree theory.

\bigskip

As a matter of fact, we mention that in \cite{BKnon94} and \cite{MZdm97}, the authors constructed a blowup solution satisfying the so-called \textit{intermediate} blowup profile, valid for $0\le t<T$:
\begin{equation}\label{bkmz}
\left\| (T-t)^{\frac{1}{p-1}}u(.,t)  - \left(p-1 + \frac{(p-1)^2}{4p} \frac{|.|^2}{(T-t)|\ln(T-t)|} \right)^{-\frac{1}{p-1}} \right\|_{L^\infty} \le \frac{C}{1 + \sqrt{|\ln(T-t)|}}.
\end{equation}
The derivation of the so-called \textit{final} profile \eqref{equiv}, valid at $t=T$,  was later done by Zaag in \cite{ZAAihn98}.

\medskip

Note that
the  constructive method given in  those works was efficiently used  
in a very large  class of parabolic equations such as in Merle and Zaag \cite{MZnon97} for   quenching   problems;
in Duong et al \cite{DNZtunisian-2017}, Nguyen and Zaag in \cite{NZsns16},   and Tayachi and Zaag \cite{TZpre15} for perturbed nonlinear source terms;  in Duong et al  \cite{DNZMAMS20,DNZIHPNA21}, Masmoudi and Zaag \cite{MZjfa08} and Nouaili and Zaag \cite{NZ2017}   for the Complex Ginzburg-Landau equation; and Duong  \cite{DJDE2019,DJFA2019}, and also in Nouaili and Zaag \cite{NZcpde15}  for non-variational complex valued heat equations.

\bigskip

More generally, a large literature has been devoted in the last 20 years to the construction of solutions of PDEs with prescribed behavior, beyond the case of parabolic equations such as: Type I anisotropic heat equation by Merle et al \cite{MRSIMRNI20}; Type II blowup for heat equation by del Pino et al \cite{PMWAMS19,PMWAPDE20,PMWQYARXIV20,PMWZDCDS20}, Schweyer \cite{SJFA12}, Collot \cite{CAPDE17}, Merle et al \cite{CMRJAMS20}, Harada \cite{HAIHPANL20,HAPDE20}, Seki \cite{SJDE20}; blowup for nonlinear Schr\"odinger equation by  Merle \cite{MCMP90}, Martel and Rapha\"el \cite{MRASENS18}, Merle et al \cite{MRRCJM15,MRSDMJ14,MRRSARXIV20}, Rapha\"el and Szeftel \cite{RSCMP09};  Blowup for wave equations by C\^ote and Zaag \cite{CZCPAM13}, Ming et al \cite{MRTSJMA15}, Collot \cite{CPMJEMS20}, Hillairet and Rapha\"el \cite{HRAPDE12}, Krieger et al \cite{KSTDM09,KSTIM08},  Ghoul et al \cite{GINJDE18}, Rapha\"el and Rodnianski \cite{RRPMIHES12}, Donninger and Sch\"orkhuber \cite{DSCMP16}; Blowup for KdV and gKdV \cite{MAJM05}, C\^ote \cite{CJFA06,CDM07}; Schr\"odinger map by Merle et al \cite{MRRim13}; Heat flow map  by Ghoul et al \cite{GINAPDE19},  Rapha\"el and  Schweyer \cite{RSAPDE14}, D\'avila et al \cite{DPMIM20}; Keller  Segel system by Ghoul et al \cite{CGMNARXIV20-a,CGMNARXIV20-b}, Schweyer and Rapha\"el \cite{RSma14}; Prandtl's system by Collot et al \cite{CGIMARXIV18}; Stefan problem by Hadzic and Rapha\"el \cite{HRJEMS19}; 3-dimensional compressible fluids by
Merle et al \cite{MRRSARXIV20};
quenching phenomena for MEMS devices by Duong and Zaag \cite{DZM3AS19};  the Gierer-Meinhardt system  by Duong et al   \cite{DGNZARXIV21,DNZarxiv2020}.

\medskip
Next,  we would like to mention some papers where upper bounds on the profile of the gradient were obtained, though with no sharp equivalent.
For example,   in Tayachi and Zaag \cite{TZpre15}, the authors constructed    a blowup solution to the following equation
$$ \partial_t u = \Delta u + \mu |\nabla u|^q  + |u|^{p-1} u, $$
where
$$ \mu > 0, p>3 \text{  and  } q = \frac{2p}{p+1}.$$
In addition to that, they  proved  that the gradient also blows up and gave an upper bound on the gradient's final profile near $0$:
$$  |\nabla u(x,T)|  \le \left\{   \begin{array}{rcl}
& &C   \frac{|x|^{-\frac{p+1}{p-1}}}{|\ln|x||^{\frac{1-3p}{(p-1)^2}-\epsilon}} \text{ if } p \in (3,7],\\[0.2cm]
& & C   \frac{|x|^{-\frac{p+1}{p-1}}}{|\ln|x||^{\frac{-p^2 +2p-5}{2(p-1)^2}-\epsilon}}  \text{ if }   p>7,
\end{array} \right.  $$
for some small $\epsilon>0$.
Similarly, such  upper bounds were  proved  in many situations such as for a     perturbed nonlinear heat equations with a gradient and a non-local term as in Abdelhedi and Zaag  \cite{AZarix20}; and nonlinear heat equations involving a critical power nonlinear gradient term as  in Ghoul et al  \cite{GNZpre16a}. 

\bigskip

In our opinion, 
it is exciting and important  to obtain the exact profile of the gradient. By sharply adapting the previously mentioned construction method to handle the gradient estimates, we get the following result:
\begin{theorem}[Construction of a blowup solution with a sharp determination of the gradient behavior]\label{Theorem-similarity} There exist initial  data $u_0 \in W^{1,\infty}(\R^N)$ such that  equation \eqref{equa-u-vortex} has a unique solution which  blows  up in finite time $T(u_0) > 0$ only at the origin. In particular, the following holds:
\begin{itemize}
\item[$(i)$] (Intermediate profile): For all $t\in [0,T)$,
\begin{eqnarray}
\left\| (T-t)^{\frac{1}{p-1}} u(.,t) -   \varphi_0\left(z \right)  \right\|_{L^\infty} \le \frac{C \ln(|\ln(T-t)|)}{1+ |\ln(T-t)|},\label{profile-u-theo}
\end{eqnarray}
and 
\begin{eqnarray}
\left\| (T-t)^{\frac{1}{p-1} +\frac{1}{2}} \nabla u(.,t) -\frac{\nabla\varphi_0(z)}{\sqrt{|\ln (T-t)|}}  \right\|_{L^\infty}  \le   \frac{C\ln(|\ln(T-t)|)}{1 + |\ln (T-t)|},\label{profile-intermediate-main-theorem} 
\end{eqnarray}
where
\begin{equation}\label{defphi0}
z=\frac{x}{\sqrt{(T-t)|\ln(T-t)|}},\;\;
\varphi_0(z) =  \left( p-1 + b |z|^2\right)^{-\frac{1}{p-1}}
\mbox{ and }b = \frac{(p-1)^2}{4p}.
\end{equation}
\item[$(ii)$] (Final profile):  $u(x,T)\equiv \lim_{t \to T} u(x,t)$ exists for all $x \ne 0$ and $ u(\cdot,T) \in C^2(\R^N \setminus \{0\})$, with a convergence holding uniformly in $C^2$ of every compact set of $\mathbb{R}^N\backslash\{0\}$. In particular, we have the following 
equivalents:
\begin{eqnarray}
u(x,T)\sim u^*(x) \equiv \left[ \frac{b}{2}  \frac{|x|^2}{|\ln|x||} \right]^{-\frac{1}{p-1}},\label{final-profile}
\end{eqnarray}
and 
\begin{equation}\label{final-profile-main-theorem}
  \nabla u(x,T)\sim \nabla u^*(x)  =   -  \frac{\sqrt{2b}}{p-1} \frac{x}{|x| \sqrt{ |\ln|x||}}    \left[\frac{b|x|^2}{2|\ln|x||} \right]^{-\frac{p+1}{2(p-1)}}  \text{ as } x \to 0.
\end{equation}
\end{itemize}
\end{theorem}

\begin{remark}
The main contribution of the theorem is to discover \eqref{profile-intermediate-main-theorem}  and its consequence  \eqref{final-profile-main-theorem}.  The key estimate is to show a more precise error estimate in \eqref{profile-u-theo}, namely
$$ \frac{C \ln|\ln(T-t)|}{1+|\ln(T-t)|},$$
which is sharper than the bound of the previous works \cite{BKnon94} and \cite{MZdm97}, namely
$$ \frac{C}{1 + \sqrt{|\ln(T-t)|}},$$
as already mentioned in \eqref{bkmz}. Indeed, this latter estimate is not enough to derive a sharp estimate on the gradient.
%
%
The corner stone of the proof lays in fact in the introduction of a new \textit{Shrinking set} below in Definition \ref{shrinking-set}, with sharper estimates.
We also mention other  situations where a clever adaptation of the shrinking set lead to the derivation of a sharper blowup behavior. This was in particular the case a  complex-valued heat equation with no variational structure, where the behavior of the imaginary part was derived by Duong in \cite{DJDE2019,DJFA2019}. We also mention the case of the  Complex  Ginzburg-Landau equation (CGL) in some critical setting in  Duong et al \cite{DNZMAMS20} and Nouaili and Zaag \cite{NZ2017}; and in the subcritical range in Duong et al \cite{DNZIHPNA21} too, where a higher order expansion was derived thanks to a good adaptation of the shrinking set.
%
\end{remark}
\begin{remark}[Stability] As we explained right before the statement of this theorem, the rigorous proof goes through the reduction of the question to a finite-dimensional problem, then the solution of this finite-dimensional problem thanks to the degree theory. In fact, it is possible to make 
the interpretation of the parameters of the finite-dimensional problem in terms of the blowup time and the blowup point, as originally done  in \cite{MZdm97}, in order to show that the behavior described by \eqref{profile-u-theo}, \eqref{profile-intermediate-main-theorem}, \eqref{final-profile} and \eqref{final-profile-main-theorem} is stable under perturbations of initial data. More precisely, let us denote by $\hat u_0$ the initial data constructed in Theorem \ref{Theorem-similarity} and $\hat T$ its blowup time. Then, there exists a neighborhood $	\mathcal{N}(u_0)$ of $u_0$ in $W^{1,\infty}$ 
such that for all $u_0 \in \mathcal{N}_0$, the corresponding solution $u(x,t)$  of equation \eqref{equa-u-vortex} blows up in finite time $T(u_0)$ at some unique blowup point $a(u_0)$, with the blowup profiles \eqref{profile-u-theo}, \eqref{profile-intermediate-main-theorem}, \eqref{final-profile} and \eqref{final-profile-main-theorem} which remain valid for $u(x,t)$, just by replacing $ x $ by $x - a(u_0)$. Moreover, we have the following limit
$$   (a(u_0),  T(u_0)) \to (0,\hat T) \text{ as } u_0 \to \hat u_0.  $$

\end{remark}

\begin{remark}
This construction works also in bounded domains with homogenous Dirichlet or Neumann boundary conditions, with almost the same proof. The only difference with the case of the whole space lays in the use of some cut-off argument, which runs smoothly, as in Mahmoudi, Nouaili and Zaag \cite{MNZNon2016} or in Duong and Zaag \cite{DZM3AS19}. 
\end{remark}





\bigskip


Throughout this paper, $C$ is a constant which depends only on $N$ and $p$, and whose value may change from line to line. If we need a constant depending on other parameters, we will specify its dependence on parameters. However, we may omit the dependence on $K$, the constant introduced  in the cut-off \eqref{c4defini-chi-y-s} below, in order to make the notations lighter.

\section{Mathematical setting}\label{math-setting}
Let us consider $T>0$ and introduce  the  following similarity variables
 \begin{equation}\label{similarity-variable}
y  = \frac{x}{\sqrt{T-t}},     \quad  s  = -\ln(T-t)  \text{ and  } w (y,s)  = (T-t)^{\frac{1}{p-1}} u(x,t).
\end{equation}
With this transformation, $u(x,t)$ satisfies equation \eqref{equa-u-vortex} for all $(x,t)\in \mathbb{R}^N\times [0,T)$ if and only if $w(y,s)$ satisfies the following equation for all $(y,s)\in \mathbb{R}^N\times[-\ln T, +\infty)$:
\begin{equation}\label{equa-w}
\partial_s w = \Delta w - \frac{1}{2} y \cdot \nabla w - \frac{w}{p-1} +|w|^{p-1}w.
\end{equation}
Since our goal is to construct a solution $u$ 	to equation  \eqref{equa-u-vortex} that blows up in finite time $T$ like $(T-t)^{-\frac{1}{p-1}}$, we may change it thanks to \eqref{similarity-variable} to the construction of a global  solution $w(y,s)$ to \eqref{equa-w}  such that 
$$\epsilon_0 \le \limsup_{s \to +\infty} \| w(s)\|_{L^\infty}  \le \frac{1}{\epsilon_0}$$
for some $\epsilon_0>0$.
In \cite{MZdm97}, the authors further specified this goal and succeeded in showing 
the existence of a solution to equation \eqref{equa-w} with the following profile: 
\begin{equation}\label{defphi}
\varphi(y,s) = \left( p-1 + b \frac{|y|^2}{s} \right)^{-\frac{1}{p-1}} + \frac{\kappa N}{2ps}, \text{ with } \kappa = (p-1)^{-\frac{1}{p-1}},
\end{equation}
in the sense that for all $s\ge -\ln T$ and $T >0$ small enough, we  have
$$ \|w(s) -  \varphi(s) \|_{L^\infty} \le \frac{C}{\sqrt{s}}.$$
In this paper, we show that the techniques of \cite{MZdm97} can be delicately improved to
get  
a smaller error term,
namely: for all $s\ge -\ln T$,
\begin{equation}\label{aim0}
 \|  w (s)  -  \varphi(s)  \|_{L^\infty(\R^N)} \le \frac{C \ln s }{s}.
\end{equation}
Introducing 
\begin{equation}\label{defq}
q  =   w - \varphi,
\end{equation}
we see from equation \eqref{equa-w} that $q$ satisfies the following equation:
\begin{equation}\label{equa-q}
\partial_s q = (\mathcal{L} + V) q  +     B(q) + R(y,s),
\end{equation}
where
\begin{eqnarray}
\mathcal{L}  &=& \Delta  - \frac{1}{2} y \cdot  \nabla + {\rm Id},\label{defi-mathcal-L}\\
V &=&   p (\varphi^{p-1} -\frac{1}{p-1}),\label{defi-V}\\
B(q) & = & \left| q + \varphi \right|^{p-1}(q + \varphi) - \varphi^p -p \varphi^{p-1} q,\label{defi-B-q}\\
R(y,s) &=& - \partial_s \varphi + \Delta \varphi -\frac{1}{2} y \cdot \nabla \varphi  - \frac{\varphi}{p-1} +|\varphi|^{p-1}\varphi.\label{defi-R}
\end{eqnarray}
Thus, from \eqref{aim0}, our aim becomes to construct a solution of equation \eqref{equa-q} such that for all $s\ge - \ln T$,
\begin{equation}\label{aim}
\|q(s)\|_{L^\infty}\le \frac{C \ln s }{s}.
\end{equation}

\medskip
In order to understand the dynamics  of  equation \eqref{equa-q}, we need to make some comments on the linear, nonlinear and remainder terms in that equation:

-    \textit{ Linear operator  $\mathcal{L}$}:  It  is self-adjoint in the space $L^2_\rho(\mathbb{R}^N)$ where $\rho(y) = e^{-|y|^2/4}/(4\pi)^{N/2}$,
with     explicit spectrum
\begin{equation}\label{spec}
 \text{Spec}(\mathcal{L}) = \{  \lambda_m=  1 - \frac{m}{2}   |    m \in \N  \}.
\end{equation}
Corresponding to  the eigenvalue $\lambda_m$, we have the eigenspace $\mathcal{E}_m$
\begin{equation}\label{eigen-space-E-j}
 \mathcal{E}_m =  \left<   h_{m_1} (y_1). h_{m_2} (y_2).... h_{m_N} (y_N)   \left|  \right.  m_1 + ...+ m_N = m    \right>, 
\end{equation}
where   $h_{\ell}$  is  the  (rescaled) Hermite  polynomial in one  dimension, defined by
$$h_\ell (\xi)= 	 \sum_{j=0}^{\left[ \frac{\ell}{2}\right]} (-1)^j\frac{ \ell! }{j!(\ell-2j)!}\xi^{\ell-2j}.$$

-  \textit{Potential $V$}:  It has the following properties:
\begin{itemize}
\item[$(i)$]   $V(., s) \to 0$  in $L^2_\rho(\mathbb{R}^N)$ as  $s \to + \infty$, which implies in particular that its effect is negligible with respect to the effect of    $\mathcal{L}$ (except maybe on the zero eigenspace of $\mathcal{L}$).
\item[$(ii)$]   $V(y,s)$ is almost a constant   outside the blowup region, i.e for  $|y| \geq K_0 \sqrt s$ with $K_0>0$ large enough. In particular,    we have the following estimate
$$     \sup_{ s \geq  s_\epsilon, \frac{|y|}{ \sqrt s}  \geq \mathcal{C}_\epsilon} \left|  V(y,s)   - \left( -\frac{p}{p-1}\right)  \right|   \leq  \epsilon, $$
for some   $\epsilon > 0$,     $ \mathcal{C}_\epsilon >  m$, and  $s_\epsilon$ large enough. Since  $ - \frac{p}{p-1 }  <   -1  $ and the largest eigenvalue of $\mathcal{L}$ is 1, we can say that  $\mathcal{L} + V$ has  a  strictly  negative  spectrum outside the blowup region.
 \end{itemize}
 
 - \textit{Nonlinear term  $B(q)$}: It is superlinear, in the sense that it satisfies the following estimate
 $$  |B(q) | \le C|q|^{\bar p}\mbox{ where } \bar p = \min(2,p)>1. $$ 
 
 - \textit{Remainder term $ R$}:  It is small, in the sense that 
 $$ \forall s\ge -\ln T,\;\; \|R(s)\|_{L^\infty}  \le  \frac{C}{s},  $$
which is natural by definition \eqref{defi-R}, since the profile $\varphi(y,s)$ defined in \eqref{defphi} is in fact an approximate solution of equation \eqref{equa-w}.
Following the decomposition in \cite{MZdm97} together with the remark  given in item $(ii)$ above,   we introduce the following decomposition,  for any     $r  \in L^\infty(\mathbb{R}^N)$:
\begin{equation}\label{c4R=R-b+R-e}
r (y)=     \chi (y,s)   r(y) + (1- \chi (y,s) )    r (y) \equiv r_b (y,s) +  r_e (y,s),
\end{equation} 
where   $\chi (y,s)$   defined by
\begin{equation}\label{c4defini-chi-y-s}
\chi (y,s)  = \chi_0 \left(  \frac{|y|}{K \sqrt s}   \right),
\end{equation}
$\chi_0 $ being a one-dimensional cut-off satisfying
\begin{equation}\label{c4defini-chi-0}
\text{Supp } (\chi_0) \subset [0,2], \quad 0 \leq \chi_0(\xi) \leq 1, \forall \xi\ge 0 \text{ and  }   \chi_0 (\xi)= 1, \forall \xi \in [0,1],
 \end{equation}
and  constant $K>0$  be chosen large enough so that various estimates hold in the proof.
%
Let us remark that
\begin{eqnarray*}
\text{Supp }(r_b(s))  & \subset & \{ |y| \leq 2 K\sqrt{s} \},\\
\text{Supp }(r_e(s)) & \subset & \{ |y| \geq K \sqrt{s}\}.
 \end{eqnarray*}
In addition, since  the set of  eigenfunctions  of $\mathcal{L}$  makes a basis of $L^2_\rho$, we can  write $r_b$ as follows
\begin{equation}\label{decom}
r_b (y,s)  =  r_0(s)  + r_1(s) \cdot y +   y^T \cdot r_2(s) \cdot y  - 2 \text{ Tr}(r_2(s))  + r_-(y,s),
\end{equation}
where  
\begin{equation}\label{c4defini-R-i}
r_i(s) =   \left(   P_\beta ( r_b(s) )  \right)_{\beta \in \mathbb{N}^N, |\beta|= i}, \forall  i \geq 0,
\end{equation}
with $P_\beta(r_b)$  being  the projection  of  $r_b$ on   the   eigenfunction    $h_\beta$ defined as follows:
\begin{equation}\label{c4defin-P-i}
P_\beta (r_b(s)) =  \int_{\mathbb{R}^N}  r_b(y,s)   \frac{h_\beta}{\|h_\beta\|_{L^2_\rho(\mathbb{R}^N)}^2}   \rho dy, \forall \beta \in \mathbb{N}^N,
\end{equation}
and 
\begin{equation}\label{c4defini-R--}
r_-(y,s) = P_-(r_b)  =   \sum_{\beta \in \mathbb{R}^N, |\beta| \geq 3}  P_\beta (r_b(s)) h_\beta(y).
\end{equation}
Note that in \eqref{decom}, we are isolating the nonnegative eigenvalues of $\mathcal{L}$ (see \eqref{spec}), since they are more delicate to handle.\\
 Finally, we have just defined the following expansion, which will be extremely useful in the proof:
\begin{eqnarray}
r (y) &  =  &  r_0(s) + r_1(s) \cdot y  +  y^T \cdot r_2(s) \cdot y  - 2 \text{ Tr}( r_2(s))   + r_-(y,s) + r_e (y,s).\label{c4represent-non-perp}
\end{eqnarray}
Note that this decomposition depends on time $s$ appearing in the cut-off  \eqref{c4defini-chi-y-s}. If ever $r(y) =r(y,s)$, then this gives a natural choice for the parameter $s$ in the cut-off. If not, then we will specify which time is selected (this is the case in particular in Lemma \ref{lembk} below).

\section{The major steps of the proof  assuming some technical results  }

Our method is in fact a refinement of the one used by Merle and Zaag in \cite{MZdm97}. For that reason, we will follow the same framework of that paper, improving the estimates (and the presentation too!). In order to keep  the paper in a reasonable length, we will insist only on the improvements, and go rather quickly when the estimates are the same as in \cite{MZdm97}.

\medskip

In order to achieve our goal in \eqref{aim}, we will in fact construct a solution of equation \eqref{equa-q} trapped in some shrinking set, which is an improved version of the sets already introduced by Bricmont and Kupiainen in \cite{BKnon94} and Merle and Zaag in \cite{MZdm97}. More precisely, this is our improved version:
\begin{definition}[Shrinking set]\label{shrinking-set}
For any $K>0$, $A>0$ and $s>0$, we define  $ V_{K,A} (s)$
as a subset of $L^\infty ( \R^N)$ satisfying 
\begin{eqnarray*}
\left|q_0(s)    \right| \le \frac{A}{s^2}, \quad   | q_{1}  (s)| \leq \frac{A}{s^2} \text{ and   }  |q_{2} (s)| \leq \frac{A^2 \ln s}{s^2}, \\
\left\| \frac{q_-(\cdot,s)}{1+|y|^3} \right\|_{L^\infty(\mathbb{R}^N)}  \leq \frac{A^6\ln s}{s^\frac{5}{2}},  \text{ and } \|q_e(\cdot,s)\|_{L^\infty(\R^N)} \leq \frac{A^7 \ln  s}{s},
 \end{eqnarray*}
 where $q_j(s), q_-(\cdot,s), $ and $ q_e(\cdot,s)$ are the components of $q(\cdot,s)$ as defined in 
 \eqref{c4represent-non-perp}. 
\end{definition}

\begin{remark}
Note that the parameter $K>0$ is involved in the decomposition \eqref{c4represent-non-perp}, through the cut-off in \eqref{c4defini-chi-y-s}. Note also that by definition,
we immediately derive that if $q \in V_{K,A}(s)$, then 
\begin{equation}\label{esti-q-V-A}
    \|q(s)\|_{L^\infty}  \le C \frac{A^7 \ln s}{s},
\end{equation}
for some $  C = C(K)>0$. Thus, as announced right before this definition, our goal becomes to find some positive  $K$, $A$ and initial time $s_0$
so we are able to construct a solution of equation \eqref{equa-q} satisfying
\[
\forall s\ge 
s_0,
\;\; q(s)\in V_{K,A}(s).
\]
\end{remark}

\medskip

Naturally, we first start by specifying the form of initial data we are taking.
As in \cite{MZdm97}, our initial data will depend on a set of parameters $(d_0,d_1) \in\R^{1+N}$, and will have a slightly modified form (this form was first used by Masmoudi and Zaag in \cite{MZjfa08}):
\begin{equation}\label{defi-initial-data}
 \psi_{K,A}(s_0,d_0,d_1,y) = \psi(s_0,d_0,d_1,y) = 
 \frac{A}{s_0^2} (d_0 + d_1 \cdot y) \chi_0\left( \frac{2|y|}{K\sqrt{s_0}}\right),
\end{equation}
where 
$\chi_0$ is defined \eqref{c4defini-chi-0} (note that we may or may not keep the parameters in our notation for initial data). 
In the following, we identify a set for the parameters $d_0$ and $d_1$, so that at time $s=s_0$, the solution is already in $V_{K,A}(s_0)$. We also estimate the other components of the decomposition \eqref{c4represent-non-perp}.
\begin{proposition}[Preparing initial data]\label{propo-initial-data} For all  $K\ge 1$ and $ A \ge 1$,
we can define 	$s_1(K,A) \ge 1$ such that 	 for all $ s_0 \ge s_1$ the following holds with $\psi=\psi(s_0,d_0,d_1)$ defined in \eqref{defi-initial-data}:
\begin{itemize}  
\item[$(i)$] There exits   $\mathcal{D}_{ s_0}     \subset   [-2,2]^{1+N}$  such that the function
\begin{eqnarray*}
\bar \Gamma :  \mathbb{R}^{1+N}   & \to &   \mathbb{R}^{1 +N}\\
(d_0, d_1)   & \mapsto   &  \left( \psi_0, \psi_1  \right)(s_0),
\end{eqnarray*}
 is  affine,  one to  one   from   $\mathcal{D}_{s_0}$   to 	$\hat{\mathcal{V}}_A (s_0),$ where 
\begin{equation}\label{c4defini-hat-mathcal-V-A}
\hat{\mathcal{V}}_A (s) =  \left[-\frac{A}{s^2}, \frac{A}{s^2} \right]^{1+N}.
\end{equation}  
In addition,
we have   
\begin{equation}\label{c4deg-Gamma-1-neq-0}
 \bar \Gamma \left|_{\partial \mathcal{D}_{s_0}}  \right.   \subset   \partial \hat{\mathcal{V}}_A (s_0) \mbox{ and } 
 {\rm deg} \left(  \bar \Gamma \left. \right|_{ \partial  \mathcal{D}_{s_0} }  \right) \neq 0.
\end{equation}
\item[$(ii)$]  For all $(d_0,d_1) \in \mathcal{D}_{s_0}$, $\psi \in V_{K,A}(s_0)$. More precisely, we have 
	\begin{eqnarray*}
\left|\psi_0    \right| \le \frac{A}{s^2_0}, \quad   | \psi_{1}  | \leq \frac{A}{s^2_0},\;\; 
|\psi_{2} | \leq \frac{ 1 }{s^2_0}  ,\\
\left\| \frac{\psi_-(\cdot)}{1+|y|^3} \right\|_{L^\infty(\mathbb{R}^N)}  \leq \frac{\ln s_0}{s_0^\frac{5}{2}},  \;\;
\psi_e(\cdot) \equiv 0,
 \end{eqnarray*}
 and
 \begin{equation}\label{estimate-initial-data}
    \left\| \psi (\cdot) \right\|_{W^{1,\infty}(\mathbb{R}^N)} \le \frac{1}{s_0},
\end{equation}
where $\psi_0, \psi_1, \psi_2, \psi_-$ and $\psi_e$ are  the components of $\psi$  as  in  \eqref{decom}.
\end{itemize}
\end{proposition}
\begin{proof}
The proof is quite similar to  \cite[Proposition 4.5]{TZpre15}, except for the gradient estimate, which follows very easily by the same techniques. For that reason, we omit the proof.
\end{proof}

\bigskip

Now, we would like to introduce the following parabolic regularity estimate, valid for solutions of equation \eqref{equa-q} starting from initial data \eqref{defi-initial-data}.
\begin{proposition}[Parabolic regularity]\label{control-nabla-q}
For all $K \ge 1$ and $A \ge 1$,  there exists $s_2(K, A) \ge 1$  such that for all  $s_0 \ge s_2$ the following holds:\\
Consider $q$  is a solution    to  \eqref{equa-q} on $[s_0,s^*]$  for some  $ s^*  >  s_0$ 
with $q(s_0)=\psi(s_0,d_0,d_1)$
defined as in \eqref{defi-initial-data} for some $ (d_0,d_1) \in \mathcal{D}_{s_0}$  given in Proposition \ref{propo-initial-data}. Assume in addition that $q(s) \in V_{K,A}(s)$ for all $s \in [s_0,s^*]$. Then, it follows that
	\begin{equation}
	\| \nabla q(s)\|_{L^\infty} \le \frac{\tilde C A^7 \ln s}{s } ,\forall s \in [s_0,s^*],
	\end{equation}
	where  $\tilde C$ depends only on $K$.
\end{proposition}
\begin{proof}
We follow the technique  first introduced by Ebde and Zaag  \cite{EZsema11} then used by Tayachi and Zaag  \cite{TZpre15}. Since this step is crucial for the improvement, we give the proof, with 2 cases, depending on the position of $s$ with respect to $s_0+1$.

- \textbf{Case 1:  $s \le s_0 +1$}. We write equation \eqref{equa-q} in its integral form 
\begin{equation}\label{integral-form-q}
q(s) =  e^{(s-s_0) \mathcal{L}} q(s_0)  + \int_{s_0}^s e^{(s-\tau) \mathcal{L}} G(\tau) d\tau,
\end{equation}
where
\begin{equation}\label{defi-G-tau}
     G(\tau) = V(\tau)q(\tau) + B(q(\tau)) + R(\tau), 
\end{equation}
and $e^{ t \mathcal{L}}$ is the semi-group corresponding to  $\mathcal{L}$ defined in \eqref{defi-mathcal-L} (see \cite{TZpre15} for more details).  We now consider $s_1' = \min(s_0+1,s^*)$  and  $ s \in [s_0,s_1']$.
Taking $s_0\ge 1$, we see that 
for all $ \tau \in [s_0,s]$, we have
$$ s_0 \le  \tau \le s   \le  s_0+1 \le 2s_0, $$
which implies that
\begin{equation}\label{estima-s-tau-case-1}
\frac{1}{s}   \le  \frac{1}{\tau} \le  \frac{1}{s_0} \le \frac{2}{s}.
\end{equation}
Now, applying the operator  $\nabla $  to  \eqref{integral-form-q}, we write
\begin{equation}\label{equa-nabla-q-}
\nabla q(s) = \nabla e^{(s-s_0)\mathcal{L}} q(s_0)  + \int_{s_0	}^s \nabla e^{(s-\tau)\mathcal{L}}  G(\tau) d\tau.
\end{equation}
Let us mention to  item $(ii)$ of Lemma 4.15 in \cite{TZpre15}  that 
\begin{equation}\label{estimate-nabla-semi-group}
\mbox{if }f\in L^\infty(\RN), \mbox{ then }
\| \nabla e^{\theta \mathcal{L}} f \|_{L^\infty}  \le \frac{C e^{\frac{\theta}{2}}}{\sqrt{1-e^{-\theta}}} \|f\|_{L^\infty},
\end{equation}
and
\begin{equation}\label{nabla-semi-le-nabla}
 \mbox{if }f\in W^{1,\infty}(\RN), \mbox{ then }
 \| \nabla e^{\theta \mathcal{L}} f \|_{L^\infty}  \le C e^{\frac{\theta}{2}} \| \nabla f\|_{L^\infty},
\end{equation}
for some $ \theta >0$. 

\noindent
\medskip
Using 
 \eqref{estimate-initial-data} and  \eqref{nabla-semi-le-nabla} with $\theta =s  - s_0 \le 1$, we obtain 
 \begin{eqnarray}
 \| \nabla  e^{(s-s_0)\mathcal{L}}q(s_0)\|_{L^\infty}   & \le &  C e^{\frac{s -s_0}{2}} \|\nabla q(s_0)\|_{L^\infty} \le    \frac C{s_0} \le  \frac{2C}{s}.
 \label{estima-nabla-q-s-0}
 \end{eqnarray}
 In addition to that, we see from \eqref{defi-G-tau} that for $\tau \in [s_0,s]$,
 \begin{eqnarray*}
 \left| G(\tau)    \right| \le  |V(\tau)q(\tau)  | + |B(q(\tau))| + |R(\tau)|.
 \end{eqnarray*}
 Since $V$ is bounded by definition \eqref{defi-V} and \eqref{esti-q-V-A} holds from the fact that $q(s) \in V_{K,A}(s)$ for all $s \in [s_0, s^*]$, it follows that
 \begin{eqnarray*}
 \left| V (\tau)q  (\tau)\right| \le \left|V(\tau) \right||q(\tau)| \le \frac{C A^7 \ln \tau}{ \tau}
 \end{eqnarray*}
 (with $C=C(K)$ here).
 Using again \eqref{esti-q-V-A} together  with Lemma \ref{lemma-V-q}, we write
 \begin{eqnarray*}
 \left| B(q(\tau)) \right|  \le  C |q(\tau)|^{\bar p} \le C \left( \frac{CA^7 \ln \tau}{\tau} \right)^{\bar p} \le C \frac{ \ln \tau}{\tau}, 
 \end{eqnarray*}
 for $s_0$ large enough, since $\bar p=\min(p,2) >1$. We also have 
 \begin{eqnarray*}
  |R(\tau)|  \le   \frac{C}{\tau} .
 \end{eqnarray*}
 Combining these estimates we deduce
 \begin{eqnarray*}
 \|G(\tau)\|_{L^\infty} \le  \frac{C A^7 \ln \tau}{\tau} \le  \frac{2CA^7 \ln s}{s}, 
 \end{eqnarray*}
 thanks to  \eqref{estima-s-tau-case-1}.
Consequently, applying \eqref{estimate-nabla-semi-group} with $\theta = s -\tau \in [0,s-s_0]\subset [0,1]$, we get  
\begin{eqnarray}
\|   \nabla e^{(s-\tau) \mathcal{L}}      G(\tau)\|_{L^\infty}    & \le &     \frac{C}{\sqrt{1 - e^{-(s-\tau)}}} \frac{A^7 \ln s}{s}.\label{esti-semi-nabla-G}
\end{eqnarray}
Finally, taking the $L^\infty$ norm of \eqref{equa-nabla-q-} and using again \eqref{estima-s-tau-case-1}, we get
\begin{eqnarray*}
\| \nabla q(s)\|_{L^\infty} & \le & \| \nabla  e^{(s-s_0)\mathcal{L}}q(s_0)\|_{L^\infty} + \int_{s_0}^s \|   \nabla e^{(s-\tau) \mathcal{L}}      G(\tau)\|_{L^\infty} d\tau \\
&\le &   \frac{C}{s} +  \frac{C A^7 \ln s}{s} \int_{s_0}^s \frac{1}{ \sqrt{1 - e^{-(s-\tau)}}} d\tau \\
& \le & \frac{C_1(K) A^7 \ln s}{s}.
\end{eqnarray*}

- \textbf{Case 2}: We consider the case $s > s_0 +1$, hence 
$s^* \ge s > s_0+1$. Taking $s_0\ge 1$, 
it follows that
$$s \ge  s_0+1\ge  2 \text{ and  } s = s-1+1 \le 2(s-1).$$ 
Note
that for all  $\tau  \in [s-1,s]$,  we have
\begin{equation}\label{esti-s-tau-2s-second}
    \frac{1}{s}  \le   \frac{1}{\tau} \le  \frac{2}{s}.
\end{equation}
Using  Duhamel's principle  with initial time $ s-1$, we have
\begin{equation*}
q(s)  = e^{ \mathcal{L}} q(s-1) + \int_{s-1}^{s}  e^{(s-\tau) \mathcal{L}} G(\tau) d\tau.
\end{equation*}
As for \eqref{equa-nabla-q-}, we get
\begin{eqnarray}
\nabla q(s)  =  \nabla e^{ \mathcal{L}} q(s-1)  + \int_{s-1}^{s}   \nabla  e^{(s-\tau) \mathcal{L}} G(\tau) d\tau.\label{nabla-s-s-1q}
\end{eqnarray}
Proceeding as for
\eqref{estima-nabla-q-s-0} and \eqref{esti-semi-nabla-G}, using in particular \eqref{estimate-nabla-semi-group}, \eqref{esti-q-V-A} and \eqref{esti-s-tau-2s-second}, we get 
\begin{eqnarray*}
\| \nabla e^{ \mathcal{L}} q(s-1)\|_{L^\infty}  & \le &  \frac{C}{\sqrt{1-e^{-1}}} \|q(s-1)\|_{L^\infty}  
\le  C\frac{A^7\log(s-1)}{s-1}
\le  2C\frac{A^7\log s}{s},\\
\|\nabla e^{(s-\tau) \mathcal{L}} G(\tau)\|_{L^\infty}  & \le &  \frac{CA^7}{ \sqrt{1-e^{-(s-\tau)}} } \frac{\ln \tau}{\tau}
  \le \frac{2CA^7}{ \sqrt{1-e^{-(s-\tau)}} }  \frac{\ln s}{s}.
\end{eqnarray*}
Using the inequalities and   \eqref{nabla-s-s-1q}, we obtain
\begin{eqnarray*}
\|   \nabla q(s)   \|_{L^\infty} & \le & \| \nabla e^{ \mathcal{L}} q(s-1)\|_{L^\infty}  + \int_{s-1}^s \|\nabla e^{(s-\tau) \mathcal{L}} G(\tau)\|_{L^\infty} d\tau   \\
& \le & \frac{2CA^7 \ln s}{s} +\frac{2CA^7 \ln s}{s} \int_{s-1}^s  \frac{1}{\sqrt{1-e^{-(s-\tau)}}} d\tau \le \frac{C_2(K)A^7 \ln s}{s}.
\end{eqnarray*}
 Finally, introducing
 $\tilde C(K) = \max(C_1,C_2)$, 
 we obtain
 $$ \| \nabla q(s)\|_{L^\infty }  \le \frac{\tilde CA^7 \ln s}{s}, \forall s \in [s_0,s^*],$$
 which concludes the proof.
\end{proof}

\noindent 
Recalling what we wrote in the remark following Definition \ref{shrinking-set} and in \eqref{defi-initial-data}, our goal is 
to  construct initial data (at time $s=s_0$) $ \psi(s_0,d_0,d_1)$, 
so that  the solution of \eqref{equa-q} exists for all $s\in[s_0,\infty)$ and satisfies
\begin{equation}\label{estima-q-in-V-A}
  q(s) \in V_{K, A}(s), \forall s \ge s_0.
\end{equation}
This is possible, as we state in the following:
\begin{proposition}[Existence of a solution to \eqref{equa-q} trapped in $V_{A, K}$]\label{propo-existence-solu-S-t}
There exist  $A, K \ge 1,$ $s_0(A,K) \ge 1$ large enough, and $(d_0,d_1) \in \R^{1+N}$ such that  the solution to equation \eqref{equa-q}, corresponding to     initial data (at $s=s_0$) $\psi(s_0,d_0,d_1)$ defined in \eqref{defi-initial-data}, 
exists for all $s\in [s_0, +\infty)$ and satisfies 
\begin{equation}\label{qVA}
  q(s) \in V_{K, A}(s),  \forall s \in [s_0,+\infty),
 \end{equation}
 where the shrinking set $V_{K, A}$ is introduced in Definition \ref{shrinking-set}.
\end{proposition}
\begin{proof}
In fact,
the proof completely  follows the  method used in \cite{BKnon94} and \cite{MZdm97} and based on 
two main arguments:
\begin{itemize}
    \item  \textit{Reduction to a finite dimensional problem}: In this step, we show that the control of $q(s)$ in the shrinking set $V_{K,A}(s)$ reduces to the control of $(q_0(s),q_1(s))$, the $(N+1)$-dimensional variable corresponding to the projection of $q(s)$ on $1$ and $y$, the expanding directions of the linear operator $\mathcal L$ \eqref{defi-mathcal-L} involved in equation \eqref{equa-q}. 
    \item \textit{  Topological argument}: The control of the $(N+1)$-dimensional variable is then performed thanks to a topological argument linked to Brouwer's lemma, where we fine-tune the $(N+1)$-dimensional variable $(d_0,d_1)$ appearing in initial data $\psi(s_0,d_0,d_1,y)$ \eqref{defi-initial-data}.
\end{itemize}
Let us now briefly present these two arguments. Consider $A \ge 1, K>1$ and $s_0(K,A)$ large enough  such that    Propositions      \ref{propo-initial-data} and \ref{control-nabla-q} hold. Other restrictions on these constants will appear below in Proposition \ref{pro-reduction- to-finit-dimensional}. Proceeding by     contradiction, we   assume that  for all $(d_0,d_1) \in \mathcal{D}_{s_0}$, there exists $s(d_0,d_1) \ge s_0$ such that  
$$ q(s(d_0,d_1)) \notin  V_{ K, A}(s(d_0,d_1)),$$
where $q_{d_0,d_1}$ (we write $q$ for simplicity) is the solution to equation \eqref{equa-q} corresponding to initial data $\psi(s_0,d_0,d_1)$ given at \eqref{defi-initial-data}.
Since $q(s_0)\in V_{K,A}(s_0)$ by item (ii) in Proposition \ref{propo-initial-data}, 
we can define $s_*(d_0,d_1)<+\infty$  such that
$$ s_*(d_0, d_1) = \sup \{ s_1 \ge s_0 \text{ such that } q(s) \in V_{K, A}(s),  \forall s \in \left[s_0, s_1 \right]    \}.  $$
From continuity in time of the solution, it follows that
$q(s_*(d_0,d_1)) \in \partial V_{K, A}(s_*(d_0,d_1))$. Here comes our first argument where we reduce the problem to a finite-dimensional one. More precisely, we have the following statement:
\begin{proposition}[Reduction in finite dimensions and transverse crossing]\label{pro-reduction- to-finit-dimensional}
There exist  parameters $K  \ge 1  $, $ A \ge 1$  
and $s_0 (A,K) \ge 1$ such that the following  property holds (in addition to Propositions \ref{propo-initial-data} and \ref{control-nabla-q}):  Assume that
\begin{itemize}
\item[$(a)$]  Initial data $q(s_0)=\psi(s_0,d_0,d_1)$, given  \eqref{defi-initial-data} for some $(d_0,d_1) \in \mathcal{D}_{s_0}$, 
\item[$(b)$] $ q(s) \in  V_{K,A }(s)   $, for all $ s \in [s_0, s_*]$ for some $s_*  \ge s_0$ with
$$  q(s_*) \in \partial   V_{K, A}(s_*).$$
\end{itemize}
  Then, we have: 
\begin{itemize}
\item[$(i)$] (\textit{Reduction to finite dimensions}): It holds that  
$(q_0(s_*),q_1(s_*)) \in \partial  \mathcal{\hat V}_A(s_*)$ defined in \eqref{c4defini-hat-mathcal-V-A}.
\item[$(ii)$] (\textit{Transverse crossing}): There exists  $\nu_0 > 0$ such that 
\begin{equation}\label{traverse-outgoing crossing}
\forall \nu \in (0,\nu_0),   (q_0,q_1)(s_* +\nu)      \notin \hat{\mathcal{V}}_{ A} (s_* +\nu),
\end{equation}
which implies  that 
$$ q(s_*+\nu) \notin V_{K, A}(s_*+\nu).   $$
\end{itemize}
\end{proposition}
Let us 
assume this
result and continue the proof of Proposition  \ref{propo-existence-solu-S-t}. We 
postpone
the proof of Proposition \ref{pro-reduction- to-finit-dimensional} later to
page \pageref{prood-of-propo-3-6} in Section \ref{section-reduction-in-finite-time}, since 
it
is long and technical.\\
Adjusting the constants $A$, $K$ and $s_0$ as suggested in Proposition \ref{pro-reduction- to-finit-dimensional}, we see that
$$   (q_0,q_1)(s_*(d_0,d_1)) \in  \partial \hat{\mathcal{V}}_A(s_*(d_0, d_1)). $$
Therefore, by definition \eqref{c4defini-hat-mathcal-V-A} of $\hat {\mathcal V}_A(s)$,  we can define $\Gamma$  as follows:
\begin{eqnarray*}
\Gamma:  \mathcal{D}_{s_0}  &\to & \partial [-1,1]^{1 +N }\\
 (d_0,d_1)  &\mapsto &  \Gamma(d_0,d_1) = \frac{s_*^2(d_0, d_1)}{A}  \left(q_0,q_1 \right) (s_*(d_0, d_1)).
\end{eqnarray*}
Note then  we have
the following properties:
\begin{itemize}
\item[$(i)$] $\Gamma$ is continuous from $\mathcal{D}_{s_0}$  to  $\partial [-1, 1]^{1+ N}$ 
thanks to
the continuity in time of $q$ on the one hand, and the continuity of $s_*(d_0, d_1)$, which is a direct consequence of the transverse crossing given in item (ii) of Proposition \ref{pro-reduction- to-finit-dimensional} above. 
\item[$(ii)$] If $(d_0,d_1)\in \partial {\mathcal D}_{s_0}$, then we see from item (i) in Proposition \ref{propo-initial-data} that $(q_0(s_0),q_1(s_0))\in \partial \hat{\mathcal
{V}}_{A}(s_0)$. Using again  the transverse crossing property, we see that $s_*(d_0,d_1)=s_0$. Hence, $\Gamma_{|\partial {\mathcal D}_{s_0}}$ is equal to $\bar \Gamma_{|\partial {\mathcal D}_{s_0}}$, where $\bar \Gamma$ was introduced in item (i) of   Proposition \ref{propo-initial-data}. In particular, $\Gamma$ has a non zero degree on the boundary.
\end{itemize}
From a corollary of Brouwer's lemma, this is a contradiction.
This concludes the proof of Proposition \ref{propo-existence-solu-S-t}, assuming that Proposition \ref{pro-reduction- to-finit-dimensional} holds. Once that proof is given in Section \ref{section-reduction-in-finite-time}, the proof will be complete.
\end{proof}
\section{The proof of Theorem  \ref{Theorem-similarity}}
This section is devoted to the proof of
Theorem \ref{Theorem-similarity}.   


\begin{proof}
$ $\\
(i) Consider the constants $K$, $A$ and $s_0$ given in Proposition \ref{propo-existence-solu-S-t}, and  $q(y,s)$ the solution of equation \eqref{equa-q} constructed there so that \eqref{qVA} holds. Thanks to \eqref{esti-q-V-A} and Proposition \ref{control-nabla-q}, it follows that
\[
\forall s\ge s_0,\;\;
\|q(s)\|_{L^\infty} + 
\|\nabla q(s)\|_{L^\infty} \le C(K,A) \frac{\ln s}s. 
\]
Introducing $T=e^{-s_0}$ then defining $w(y,s)$ and $u(x,t)$ by \eqref{defq} and \eqref{similarity-variable}, we easily see from Section \ref{math-setting} that $u(x,t)$ is a solution to equation \eqref{equa-u-vortex} that exists for all $(x,t)\in \R^N \times[0,T)$ such that 
\eqref{profile-u-theo} and \eqref{profile-intermediate-main-theorem}
hold.
In particular, we see that
\[
u(0,t)\to \infty \mbox{ and }
|\nabla u(z_0\sqrt{(T-t)|\ln(T-t)|},t)| \to \infty
\mbox{ as }t\to T,
\]
for any $z_0\neq 0$,
hence, the origin is a blow-up point both for $u$ and $\nabla u$. The fact that neither $u$ nor $\nabla u$ blows up outside the origin is in fact a consequence of item (ii).\\
(ii) The existence of $u^* \in C^2(\R^N \setminus \{0\})$ such that $u(\cdot,t) \to u^*(\cdot)$ uniformly in $C^2$ of any compact set of $\R^N\backslash \{0\}$ follows 
from the classical argument  given in
Merle
\cite{Mercpam92}. It remains  to   give the proofs 
of
\eqref{final-profile} and \eqref{final-profile-main-theorem}.  Let us consider    $x_0 \neq 0, |x_0| \le \epsilon_0$ small enough.
We then
define
\begin{equation}\label{defi-mathcal-U-V}
  \left\{  \begin{array}{rcl}
     \mathcal{U}(x_0, \xi, \tau )  &=& (T-t(x_0))^\frac{1}{p-1} u( x_0 + \xi \sqrt{ T-t(x_0)}, t(x_0) + \tau (T-t(x_0))),\\
\mathcal{V}(x_0, \xi, \tau) & =&  \nabla_\xi \mathcal{U}(x_0,\xi,  \tau).
   \end{array}
   \right.
\end{equation}
Consider some $K>0$.
We remark that  once $\varepsilon_0$ is fixed
small enough,  there    uniquely  exists $t(x_0)<T$ such that 
\begin{equation}\label{defi-t-x-0}
    |x_0|  = K \sqrt{(T-t(x_0)) | \ln (T-t(x_0))|} .  
\end{equation}
It is 
obvious that
\begin{equation}\label{t(x)-to-T}
    t(x_0) \to T \text{ as } x_0 \to 0.
\end{equation}
Applying $t= t(x_0)$ to         \eqref{profile-u-theo} and \eqref{profile-intermediate-main-theorem}, we  derive from the definition \eqref{defphi0} of the profile $\varphi_0$ that
\begin{eqnarray}
\left|  \mathcal{U}(x_0, \xi, 0)     - \left( p-1 + b K^2\right)^{-\frac{1}{p-1}}       \right| &\leq &     \frac{C}{| \ln( T-t(x_0))|^\frac{1}{4}},\label{esti-mathcal-U-tau=0}\\
\left|  \mathcal{V}(x_0, \xi, 0)      +\frac{x_0}{|x_0|} \frac{2bK}{(p-1)\sqrt{|\ln(T-t(x_0))| }}  \left( p-1  + b K^2\right)^{-\frac{p}{p-1}}       \right| &\leq &     \frac{C}{| \ln( T-t(x_0))|^\frac{3}{4}},\label{estima-mathcal-V-tau=0}
\end{eqnarray}
for all $|\xi| \leq  |\ln(T-t(x_0))|^\frac{1}{4}$. Next, let us define
for $t in [0,1)$
\begin{eqnarray}
\hat{\mathcal{U}}(\tau) &= &\left( (p-1)(1 -\tau) + b K^2\right)^{-\frac{1}{p-1}}, \label{hat-U-tau}\\
\hat{\mathcal{V}}(\tau) &=& - \frac{x_0}{|x_0|} \frac{2 b K}{(p-1) |\ln(T-t(x_0))|^\frac{1}{2}} \left( (p-1)(1 -\tau)  + b K^2\right)^{-\frac{p}{p-1}} .\label{hat-V-tau}
\end{eqnarray}
In particular,  these functions  solve the following system
\begin{align}
\hat{\mathcal{U}'} (\tau)  &= \hat{ \mathcal{U}}^p(\tau),\label{equhu}\\
\hat{\mathcal{V}'}(\tau)  &= p \hat{ \mathcal{U}}^{p-1}(\tau)  \hat{\mathcal{V}} (\tau)\nonumber,
\end{align}

\medskip
Our  aim is to prove the following estimate
\begin{eqnarray}
\sup_{|\xi| \leq \frac{1}{4} |\ln (T-t(x_0))|^\frac{1}{4}, \tau \in [0,1)}\left|  \mathcal{U}(x_0, \xi, \tau )  -  \hat{\mathcal{U}}(\tau)   \right| & \leq &\frac{C}{ 1 + |\ln (T-t(x_0))|^\frac{1}{4}},\label{estima-pro-mathcam-U-vortex}\\
\sup_{|\xi| \leq \frac{1}{16} |\ln (T-t(x_0))|^\frac{1}{4}, \tau \in [0,1)}\left|  \mathcal{V}(x_0, \xi, \tau )  -  \hat{\mathcal{V}}(\tau)   \right| &\leq &  \frac{C}{1 + |\ln (T-t(x_0))|^\frac{3}{4}}.\label{estima-pro-mathcam-V-vortex}
\end{eqnarray}

- \textit{Proof of  \eqref{estima-pro-mathcam-U-vortex}}: 
We don't give the details, since
the result is quite the same as 
in 
\cite[page 1564]{DJFA2019}. 
In particular,
it is independent 
from
\eqref{estima-pro-mathcam-V-vortex}.
Roughly speaking, up to some intricate cut-off estimates, we are simply using the continuity with respect to initial data for solutions of equation \eqref{equa-u-vortex}, since ${\mathcal U}$ is indeed a solution of that equation, thanks to the scaling invariance, and so does $\hat{\mathcal U}$, thanks to \eqref{equhu}.

\medskip
-\textit{Proof of  \eqref{estima-pro-mathcam-V-vortex}}:
This step is new, though the idea behind it is the same as for \eqref{estima-pro-mathcam-U-vortex}: the continuity with respect to initial data of the PDE obtained by differentiation of \eqref{equa-u-vortex}, up to some cut-off estimates. However, one should bear in mind that without the sharp improvement we have made in the definition of the shrinking set (see Definition \ref{shrinking-set} above), this step is impossible to carry on. As a matter of fact, this is the very place where one understands how crucial is our contribution in this paper.\\
Using  \eqref{equa-u-vortex} and definition \eqref{defi-mathcal-U-V},  we have
\begin{eqnarray}
\partial_\tau \mathcal{U}    = \Delta \mathcal{U} + \left|\mathcal{U} \right|^{p-1} \mathcal{U},\label{equa-mathcal-U}
\end{eqnarray}
and since \eqref{estima-pro-mathcam-U-vortex} ensures that $\mathcal{U} >0$ on   $\{ |\xi| \leq \frac{1}{4} |\ln (T-t(x_0))|^\frac{1}{4}\}  $, then,
we readily derive the following PDE satisfied by 
$\mathcal{V}$:
\begin{equation}\label{equa-mathcal-V}
    \partial_\tau {\mathcal V} = \Delta {\mathcal V} + p \mathcal{U}^{p-1} \mathcal{V}.
\end{equation}
By using the parabolic regularity for equation \eqref{equa-mathcal-V}, together with \eqref{estima-mathcal-V-tau=0} and  \eqref{estima-pro-mathcam-U-vortex}, we derive the 
following
rough estimate:
\begin{eqnarray}
\sup_{|\xi| \le \frac{1}{4} |\ln(T-t(x_0))|^\frac{1}{4}, \tau \in [0,1) } \left|\mathcal{V}(x_0,\xi,\tau ) \right|   \le \frac{C(K)}{\sqrt{|\ln(T-t(x_0))|}}.\label{rough-mathcal-V-tau-0-1}
\end{eqnarray}
   We now introduce 
$$    \bar { \mathcal{V}} (x_0, \xi, \tau)  = \mathcal{V}( x_0, \xi, \tau )  - \hat{\mathcal{V}} (\tau),       $$
and
$$ \tilde{  \mathcal{V}}_2 (x_0, \xi, \tau )= \chi_1(\xi) \bar{ \mathcal{V}} (x_0,\xi,\tau), \text{ where } \chi_1(\xi) = \chi_0 \left(\frac{16 |\xi|}{|\ln(T-t(x_0))|^\frac{1}{4}} \right), $$
and $\chi_0$ defined 
in  \eqref{c4defini-chi-0}. Note that $\mathcal{V}$ is a vector 
with
$     \mathcal{V}  =    (\mathcal{V}_1,...,\mathcal{V}_N).     $
 We can write the following equation satisfied by $ \tilde{\mathcal{V}}_2$: 
\begin{equation}\label{equa-mathcal{V}}
 \partial_\tau \mathcal{\tilde V}_2  = \Delta  \mathcal{\tilde V}_2   + p\mathcal{U}^{p-1} \mathcal{\tilde V}_2 +  G(\xi,\tau),     
\end{equation}
where
\begin{eqnarray*}
G(\xi, \tau) & = &  \Delta \chi_1 \mathcal{V} + \Delta \chi_1 \hat{\mathcal{V}} - 2( \text{div}(\nabla \chi_1 \mathcal{V}_1),..., \text{div}(\nabla \chi_1 \mathcal{V}_N) ) +p \chi_1 \hat{\mathcal{V}} (  \mathcal{U}^{p-1} - \hat{\mathcal{U}}^{p-1}(\tau) ) .
  \end{eqnarray*}
  As a matter of fact, using \eqref{rough-mathcal-V-tau-0-1} and $\chi_1$'s definition, we have 
  \begin{eqnarray*}
| \nabla \chi_1 \mathcal{V}_j|  + |\Delta \chi_1 \mathcal{V}| & \leq & \frac{C}{|\ln(T-t(x_0))|^\frac{3}{4}}, \text{ for 
any
 } j 
 =1,\dots,
 N,\label{estimat-Delta-chi-0-nala-chi-0}
  \end{eqnarray*}
  In addition to that,  we derive from  \eqref{estima-pro-mathcam-U-vortex} and \eqref{hat-V-tau} that 
 \begin{eqnarray*}
  \left|    \chi_1 \hat{\mathcal{V}} (  \mathcal{U}^{p-1} - \hat{\mathcal{U}}^{p-1}(\tau) ) \right| \leq \frac{C}{|\ln(T-t(x_0))|^\frac{3}{4}}.
\end{eqnarray*} 
Recall the facts  that for all $s -\sigma >0$,$g \in L^\infty$ and $\vec{h}=(h_1,...,h_N), h_i \in L^\infty$, we have 
$$  \|e^{(s-\sigma) \Delta} g \|_{L^\infty} \le \|g\|_{L^\infty} \text{ and } \| e^{(s-\sigma) \Delta} \text{div}(\vec{h})\|_{L^\infty}  \le \frac{C}{\sqrt{s-\sigma}} \|\vec{h}\|_{L^\infty}, $$
where $\|\vec{h}\|_{L^\infty} = \sum_{i=1}^N \|h_i\|_{L^\infty}$. Then, it follows from \eqref{estimat-Delta-chi-0-nala-chi-0} that   
\begin{equation}\label{estimate-simi-group-G}
\| 	e^{(s -\sigma)\Delta} G \|_{L^\infty}  \le \frac{C}{|\ln(T-t(x_0))|^\frac{3}{4}} \left( 1 + \frac{1}{\sqrt{s -\sigma}}  \right) \text{ for all } s -\sigma > 0.
 \end{equation}

  \medskip
  \noindent
Now, we rewrite \eqref{equa-mathcal{V}} under integral form
$$ \mathcal{\tilde{V}}_2(\tau) = e^{\tau \Delta} \mathcal{\tilde{V}}_2(0) + \int_0^\tau e^{(s-\sigma)\Delta} ( p \mathcal{U}^{p-1} \mathcal{\tilde{V}}_2 + G) d\sigma.$$
 Taking 
 the
 $L^\infty$ norm,
 then using \eqref{estima-mathcal-V-tau=0} and \eqref{estimate-simi-group-G}, together with the definitions of $\tilde{\mathcal{V}}_2$ and $\bar{\mathcal{V}}$, we obtain
 \begin{eqnarray*}
 \|\mathcal{V}_2(\tau)\|_{L^\infty}    \leq  \frac{C}{|\ln(T-t(x_0))|^\frac{3}{4}} + C\int_0^\tau \|\mathcal{V}_2(\sigma)\|_{L^\infty} d\sigma.
 \end{eqnarray*}
Thanks to Grönwall's   lemma, we derive that
$$\|\mathcal{\tilde V}_2(\tau)\|_{L^\infty}   \leq  \frac{C(K_0)}{| \ln(T-t(x_0))|^\frac{3}{4}}, \text{  for all } \tau \in [0, 1),$$
which  concludes \eqref{estima-pro-mathcam-V-vortex}.

\noindent
\medskip
In addition to  that,  we 
apply a
parabolic  regularity 
technique
to  \eqref{equa-mathcal-U} and \eqref{equa-mathcal-V}  (see more in \cite{TZpre15})
to
get 
\begin{eqnarray*}
  \forall   \tau \in  \left[\frac{1}{2},1 \right) \text{ and } |\xi| \leq  \frac{1}{16} |\ln(T-t(x_0))|^\frac{1}{4},    |\partial_\tau\mathcal{U}(x_0, \xi,\tau)| + |\partial_\tau \mathcal{V}(x_0,\xi,\tau)| \leq C,
\end{eqnarray*}
which ensures  the existence of  $ \lim_{\tau \to 1}\mathcal{U}(x_0,0,\tau)  $ and $\lim_{\tau \to 1}\mathcal{V}(x_0,0,\tau) $. We now give the final asymptotic behaviors to  $u^*$ and $\nabla_x u^*$:  From  \eqref{defi-t-x-0} and \eqref{t(x)-to-T}, we get   
\begin{eqnarray*}
\ln(T-t(x_0))  \sim   2 \ln|x_0| \text{ and }
T-t(x_0)   \sim   \frac{1}{K^2}   \frac{|x_0|^2}{2 |\ln |x_0||},
\end{eqnarray*}
as $x_0\to 0$.
Taking $\xi =0$
and making $\tau \to 1$ in
 \eqref{estima-pro-mathcam-U-vortex} and  \eqref{estima-pro-mathcam-V-vortex}, we obtain
\begin{eqnarray*}
u(x_0,T) &=& (T-t(x_0))^{-\frac{1}{p-1}} \lim_{\tau \to 1}\mathcal{U}(x_0,0,\tau) \sim  \left[  \frac{b}{2} \frac{|x_0|^2}{|\ln|x_0||}\right]^{-\frac{1}{p-1}},\\
\nabla u(x_0,T) &=&  (T-t(x_0))^{-\frac{1}{p-1}-\frac{1}{2}}\lim_{\tau \to 1}\mathcal{V}(x_0,0,\tau) \sim -\frac{\sqrt{2b}}{p-1} \frac{x_0}{|x_0|} \frac{1}{\sqrt{|\ln|x_0||}}\left[  \frac{b}{2} \frac{|x_0|^2}{|\ln|x_0||}\right]^{-\frac{1}{p-1}-\frac{1}{2}},
\end{eqnarray*}
as $x_0 \to 0$.
This concludes the proof 
of   Theorem \ref{Theorem-similarity}.
\end{proof}

\section{Reduction in finite dimensions and transverse crossing}
\label{section-reduction-in-finite-time}
We prove Proposition \ref{pro-reduction- to-finit-dimensional} in this section. 
We proceed in two parts:\\
- In Part 1, we understand the dynamics of equation \eqref{equa-q} inside the new shrinking set given in Definition \ref{shrinking-set}. This is in fact the heart of our argument.\\ 
- In Part 2, we conclude the proof of Proposition \ref{pro-reduction- to-finit-dimensional}.

\medskip

\textbf{Part 1: Dynamics of equation \eqref{equa-q}}

In this 
part,
we follow the dynamics of each component of the solution, as introduced in the expansion \eqref{c4represent-non-perp}. This is our statement:
\begin{proposition}
[Dynamics of equation \eqref{equa-q}]
\label{propo-ODE-mode-V-A-s}
There exist  $K_3 \geq 1$ and $ A_3 \geq 1$ such that for  all $ K \geq K_3, A \geq A_3$ and $ \ell^*>0   $,   there exists    $s_3(A,K,\ell^*) > 1$ such that for all $\sigma \ge s_0 \ge s_3$,     the following property is valid:\\
Consider $q(y,s)$ a solution of equation \eqref{equa-q} 
such that $ q(s) \in V_{K, A}(s) $, for all $s \in [\sigma, \sigma +\ell]$ for some $\ell\in[0,\ell^*]$. If $\sigma=s_0$, we further assume that $q(y,s_0) = \psi(s_0,d_0,d_1,y)$ defined in  \eqref{defi-initial-data}, for some $(d_0,d_1) \in \mathcal{D}_{s_0}$ defined  in Proposition \ref{propo-initial-data}.\\
	Then, for all $s \in [\sigma, \sigma + \ell]$,
	the following properties hold: 
\begin{itemize}
\item[$(i)$] (Unstable modes
$q_0$ and $q_1$):
\begin{eqnarray*}
\left|  q_0'(s)  - q_0 (s)  \right|  \leq  \frac{\bar C}{s^2} \text{ and }
\left|  q'_1(s) - \frac{1}{2}  q_1 (s)  \right| \leq  \frac{\bar C}{s^2},
\end{eqnarray*}
for some constant $\bar C>0$.
\item[$(ii)$]  (Neutral mode $q_2$):
\begin{equation}
 \left|   q_2 '(s)   + \frac{2 }{s} q_2(s)   \right| \leq \frac{CA }{s^3}, \label{estima-priori-q-2}
\end{equation}
\item[$(iii)$]  \text{(Control of negative  spectrum part  $q_- $ and  outer part $q_e$)}:

- If  $\sigma > s_0$, then we have 
\begin{eqnarray}
\left\| \frac{q_-(.,s)}{1+|y|^3} \right\|_{L^\infty} \hspace{-0.3cm}  &\leq & \frac{C \ln s}{s^\frac{5}{2}} \left( A^6 e^{-\frac{s-\sigma}{2}} + A^7 e^{-(s-\sigma)^2}  + A^2 e^{(s-\sigma)}((s-\sigma)^2+1)   + s -\sigma \right) \label{estima-q-1-y-3-P-1-stric},\\[0.2cm]
\|q_e(.,s)\|_{L^\infty} \hspace{-0.3cm} & \leq & \frac{C\ln s}{s} \left(    A^7 e^{-\frac{s-\sigma}{p}}  +  A^6 e^{s-\sigma}   +   s-\sigma  \right) \label{estina-q-e--P-1}.
\end{eqnarray}

- If  $\sigma = s_0$, then we have 
\begin{eqnarray}
\left\| \frac{q_-(.,s)}{1+|y|^3} \right\|_{L^\infty}  &\leq & \frac{C\ln s}{s^\frac{5}{2}} \left( 1  + s -s_0  \right) \label{estima-q-1-y-3-P-1-s-0},\\[0.2cm]
\|q_e(.,s)\|_{L^\infty} & \leq & \frac{C\ln s}{s} \left(1 + s-s_0  \right) \label{estina-q-e--P-1-s-0}.
\end{eqnarray}
\end{itemize}
\end{proposition}

\begin{proof}
The proof is quite the same as     \cite[Lemma 4.1]{DNZtunisian-2017}, except for the bounds involving 
$q_-$ and $ q_e$
which are substantially
improved.
Note that
the improvement allows us to   describe   the 
final profile 
of the gradient in Theorem \ref{Theorem-similarity}. 
Let us
mention in addition
that
the result on $q_-$ and $q_e$ follows
from
the result in  \cite[Lemma 2.7]{NZens16}  
(see below) combined 
with the new bounds in the shrinking set $V_{K, A}(s)$. For 
that
reason, we omit  the details of the proofs relating to  to the finite-dimensional modes $q_0, q_1$ and $q_2$,  and 
focus on the proof of 
the
estimates on  $ q_- $and $q_e$.  By Duhamel's principle,   equation     \eqref{equa-q}   leads to the following integral 
form
\begin{eqnarray}
q(s) = \mathcal{K}(s,\sigma)(q(\sigma))    + \int_\sigma^s  \mathcal{K}(s,\tau) \left( B(q) + R \right)(\tau)  d\tau,\label{integral-equation}
\end{eqnarray}
where   $\mathcal{K}(s,\sigma)$ is the fundamental  solution  associated  to the linear  operator  $\mathcal{L} + V$ defined in \eqref{defi-mathcal-L} and \eqref{defi-V}.

\medskip

Let us recall 
the following fundamental result whose proof goes back to Bricmont and Kupiainen \cite{BKnon94}, though it was not stated there in this form. In fact, we owe the following form to Nguyen and Zaag \cite{NZens16} (see Lemma 2.7 in that paper):
\begin{lemma}[Bricmont-Kupiainen \cite{BKnon94} ;
Dynamics of the linear operator ${\mathcal L}+V$]
\label{lembk}
There exists $K^*\ge 1$ such that for all $K\ge K^*$ and
$\ell^* > 0$,  there exists  $s^*(K,\ell^*)$ such that  for all $ \sigma \ge s^*$ and 
$v\in L^\infty$
satisfying 
$$ \sum_{j=0}^2 |v_j(\sigma)| +  \left\|\frac{v_-(\sigma)}{1 +|y|^3} \right\|_{L^\infty} + \|v_e(\sigma)\|_{L^\infty} < +\infty,   $$
where the components  defined as in \eqref{c4represent-non-perp} (with $\chi(y,\sigma)$ as a cut-off function in \eqref{c4defini-chi-y-s}), it holds that for any $s \in [\sigma, \sigma +\ell^*]$, 
\begin{eqnarray}
\left\|\frac{\theta_-(y,s)}{1+|y|^3} \right\|_{L^\infty }   & \le &  C	\frac{e^{s-\sigma} ((s-\sigma)^2+1) }{s} (|v_0(\sigma)| + |v_1(\sigma)| +  \sqrt{s} |v_2(\sigma)| ) \label{theta--}\\
&+& Ce^{-\frac{s -\sigma}{2}} \left\|\frac{v_-(\sigma)}{1 + |y|^3} \right\|_{L^\infty}  + C \frac{e^{-(s-\sigma)^2}}{s^\frac{3}{2}} \|v_e(\sigma)\|_{L^\infty} \nonumber,
\end{eqnarray}
and 
\begin{eqnarray}
\| \theta_e(s)\|_{L^\infty}  \le Ce^{s-\sigma} \left( \sum_{j=0}^2 s^{\frac{j}{2}} \left|v_j (\sigma)\right|   + s^{\frac{3}{2}} \left\|\frac{v_-(\sigma)}{1 +|y|^3} \right\|_{L^\infty} \right)  + Ce^{-\frac{s-\sigma}{p}} \|v_e(\sigma)\|_{L^\infty},\label{theta-e}
\end{eqnarray}
where $\theta(s)  = \mathcal{K}(s,\sigma) v$.
\end{lemma}
\begin{proof}
See Lemma 2.7 in Nguyen and Zaag \cite{NZens16},
\end{proof}
Since $q(s)$ appears in \eqref{integral-equation} as a sum of 2 terms, we proceed in 3 steps, with step 1 and 2
dedicated to each term. We then conclude the proof in Step 3. As the fundamental solution ${\mathcal K}(s,\sigma)$ is involved in both of them, Lemma \ref{lembk} will be crucial here.\\
Let us first take $K\ge K^*$ and $\sigma \ge s_3(K,\ell^*) \ge s^*(K,\ell^*)$ large enough (where $K^*$ and $s^*$ are defined in Lemma \ref{lembk}), so that 
\begin{equation}\label{estimate-1-s-1-tau}
\mbox{if } \sigma \le \tau\le s \le \sigma+\ell^*,
\mbox{ then }
  \frac{1}{s}\le \frac{1}{\tau}  \le \frac{1}{\sigma} \le \frac{2}{s}. 
\end{equation}

\medskip

\textbf{Step 1: The estimate for  $\mathcal{K}(s,\sigma)q(\sigma)$}

Using the fact that $q(\sigma) \in V_{K,A}(\sigma) $ and \eqref{estimate-1-s-1-tau}, we get 
\begin{eqnarray*}
& & \left| q_0(\sigma) \right|  \le \frac{A}{\sigma^2} \le \frac{CA}{s^2}, \quad  \left| q_1(\sigma) \right| \le \frac{A}{\sigma^2} \le \frac{CA}{s^2}, \quad \left| q_2(\sigma ) \right| \le \frac{A^2 \ln \sigma }{\sigma^2} \le \frac{C A^2 \ln s}{s^2}, \\
& &  \left\|\frac{q_-(\cdot, \sigma)}{1 + |y|^3} \right\|_{L^\infty} \le \frac{A^6 \ln \sigma}{\sigma^\frac{5}{2}} \le \frac{C A^6 \ln s}{s^\frac{5}{2}}, \text{ and } \|q_e(\sigma)\|_{L^\infty} \le \frac{A^7 \ln \sigma}{\sigma} \le \frac{C A^7 \ln s}{s}.
\end{eqnarray*}
 Then, 
 applying
 \eqref{theta--}  and \eqref{theta-e}  with  $ v =q(\sigma)$  and $\theta(s) = \mathcal{K}(s, \sigma) v$, we see that
\begin{eqnarray}
\left\| \frac{\theta_{-}(\cdot,s)}{1+|y|^3} \right\|_{L^\infty} & \le &    \frac{C\ln s}{s^\frac{5}{2}} \left(   A^6 e^{-\frac{s-\sigma}{2}} + A^7 e^{-(s-\sigma)^2} + A^2 e^{(s-\sigma)} ((s-\sigma)^2+1)      \right),\label{Ky-}\\
\|\theta_{e}(s)\|_{L^\infty} & \le & \frac{C \ln s}{s} \left( A^7 e^{-\frac{s-\sigma}{p}} + A^6 e^{(s-\sigma) }     \right)\label{Kye}.
\end{eqnarray} 
Now, if
$ \sigma = s_0$, we get 
better bounds for $ q(s_0)  $ 
from
item $(ii)$ of Proposition \ref{propo-initial-data}, 
improving the estimates
as follows:
\begin{eqnarray}\label{Kys0}
\left\| \frac{\theta_{-}(\cdot,s)}{1+|y|^3} \right\|_{L^\infty}  \le   \frac{C\ln s}{s^\frac{5}{2}} \text{ and } 
\|\theta_{e}(s)\|_{L^\infty}  \le  \frac{C \ln s}{s}.
\end{eqnarray} 

\medskip

\textbf{Step 2: The estimate 
on
$\int_\sigma^s  \mathcal{K}(s,\tau) \left( B(q) + R \right)(\tau)  d\tau $}

Here, we need some technical and straightforward estimates on $B(q)$ and $R$, which are stated in Lemma \ref{lemma-V-q}. For that reason, we will take $\sigma \ge s_3'(K,A,\ell^*)$ for some $s_3'=\max(s^*(\ell^*),s_3(\ell^*), s_5(K,A))$, where $s^*$, $s_3$ and $s_5$ are defined respectively in Lemma \ref{lembk}, 
right before \eqref{estimate-1-s-1-tau} and 
in Lemma \ref{lemma-V-q}.
Then, we consider $\ell\in[0,\ell^*]$, $s\in[\sigma, \sigma + \ell]$ and $\tau \in [\sigma, s]$. Note in particular that \eqref{estimate-1-s-1-tau} holds here.\\
Applying item $(iii)$ of Lemma \ref{lemma-V-q} together with straightforward standard techniques, we see that
\begin{equation}\label{R}
\|R(\cdot, \tau)\|_{L^\infty} \le \frac{C}{\tau} \le \frac{2C }{s},
\;\;
\|R_e(\cdot, \tau)\|_{L^\infty} \le \frac{4C }{s}.
\end{equation}
\begin{equation}\label{R2}
    \left| R(y,\tau) - \frac{a_0}{\tau^2}  \right| \le \frac{C(1 + |y|^3)}{s^3}, \forall |y| \le 2K \sqrt{s}, \left| R(y,\tau) \right| \le \frac{C(1 +|y|^3) }{\tau^\frac{5}{2}} \le \frac{C(1 +|y|^3) }{s^\frac{5}{2}},  \forall |y| \ge K \sqrt{s},
\end{equation}
which yields
\begin{eqnarray*}
\left| R_-(y,\tau) \right|  \le \frac{C(1 +|y|^3) }{s^\frac{5}{2}}, \forall y \in \R^N .
\end{eqnarray*}
\noindent
%
Now, we focus on the estimate for $B(q)$.
Using
the fact  that 
$q(\tau) \in V_{K, A}(\tau)$ 
by hypothesis, we see from \eqref{esti-q-V-A} together with straightforward computations based on Definition \ref{shrinking-set} of $V_{K,A}$
that 
\begin{eqnarray*}
 \|q(\cdot, \tau)\|_{L^\infty} \le \frac{C A^7 \ln s}{s} ,
\end{eqnarray*}
and 
\begin{eqnarray*}
\left| q(y,\tau)  \right|  \le C \left( \frac{A}{s^2}(1 + |y|) + \frac{A^2 \ln s}{s^2} (1 + |y|^2) +  \frac{A^6 \ln s}{s^\frac{5}{2}} (1 + |y|^3) +
\mathbbm{1}_{\{|y|\ge K \sqrt{\frac s2}\}}
\frac{A^7 \ln s}{s} \right). 
\end{eqnarray*}
Let us define $f = B(q)$. 
Applying
item $(ii)$ of Lemma \ref{lemma-V-q} with  $\bar p =\min(p,2) >1$, 
we see
that 
\begin{equation}\label{fcs}
|f|   \le C |q|^{\bar p} \le C \left( \frac{A^7 \ln s}{s } \right)^{\bar p} \le \frac{C}{s}
\end{equation}
for $s$ large enough. In addition to that,
for all $ |y| \le 2K \sqrt{\tau}$, $(ii)$ of Lemma \ref{lemma-V-q}  ensures that
\begin{eqnarray}
& & \left| f (y)  \right| \le C(K)|q(y,\tau)|^2\nonumber\\
& \le & C(K) \left( \frac{A^2}{s^4}(1 +|y|^2)    + \frac{A^4 \ln^2 s}{s^4} (1 + |y|^4) +\frac{A^{12} \ln^2 s}{s^5} (1 + |y|^6)+ 
\mathbbm{1}_{\{|y|\ge K \sqrt\frac s2\}}
\left( \frac{A^7\ln s}{s} \right)^2 \right)\nonumber \\
& \le &   \frac{C(K) (1 + |y|^3)}{s^\frac{5}{2}}\label{f1}
\end{eqnarray}
for $s$ large enough,
since 
\begin{eqnarray*}
\frac{|y|}{\sqrt{s}} \le  \frac{|y|}{\sqrt{\tau}} \le 2K \text{ if } |y| \le 2K\sqrt{\tau} \text{ and }
\frac{|y|}{\sqrt{s}}  \ge \frac K{\sqrt 2}
\text{ if } |y| \ge K 
\sqrt{\frac s2}.
\end{eqnarray*}
On the other hand, when $ |y| \ge K \sqrt{\tau}$, we derive from \eqref{estimate-1-s-1-tau} that 
\begin{eqnarray*}
\frac{1 + |y|^3}{s^\frac{3}{2}} \ge C^{-1}K^3,
\end{eqnarray*}
hence, using \eqref{fcs}, we see that
\begin{eqnarray}
\left|  f (y)\right|  \le 
\frac Cs
\frac{1+ |y|^3}{s^\frac{3}{2}}   \le \frac{C(1 + |y|^3)}{s^\frac{5}{2}}. \label{f2}
\end{eqnarray}
Finally, 
using \eqref{fcs}, \eqref{f1} and \eqref{f2}, we see that
\begin{eqnarray}\label{bq}
\left| B(q)(y, \tau) \right| \le \frac{C (1 +|y|^3)}{s^\frac{5}{2}} \text{ and } \left| B(q)(y,\tau) \right| \le \frac{C}{s}.
\end{eqnarray}
Let us define $v(\tau) = B(q)(\tau) + R (\tau) $  and $\theta(s,\tau) = \mathcal{K}(s,\tau) v(\tau)  $,
using \eqref{R}, \eqref{R2} and \eqref{bq} together with \eqref{estimate-1-s-1-tau}, we see that
\begin{eqnarray*}
\left| v(y,\tau) - \frac{a_0}{\tau^2} \right|  \le \frac{C(1 +|y|^3)}{s^\frac{5}{2}}, \forall |y| \le 2K \sqrt{\tau}, \left| v(y,\tau)  \right| \le \frac{C(1 +|y|^3)}{s^\frac{5}{2}}, \forall |y| \ge K \sqrt{\tau},   \text{ and }  \| v(\tau) \|_{L^\infty} \le \frac{C}{s}.
\end{eqnarray*}
\\
Then, by definition \eqref{c4represent-non-perp} of the components of $v$, we have the following 
estimates:\\
\begin{eqnarray}\label{vVA}
\left| v_0(\tau) \right| \le \frac{C}{s^2},   |v_m(\tau)| \le \frac{C}{s^\frac{5}{2}}, m=1,2 \text{ and } |v_-(y,\tau)| \le \frac{C(1 +|y|^3)}{s^\frac{5}{2}} \text{ and } |v_e(y,\tau)| \le \frac{C}{s}.
\end{eqnarray}
Recalling that $s-\tau \le s-\sigma \le \ell \le \ell^*$, then, 
applying  \eqref{theta--}   and \eqref{theta-e},
we see that
\begin{eqnarray}
\left| \theta_{-}(y,s,\tau) \right| \le  \frac{C\ln s}{s^\frac{5}{2}}(1 +|y|^3) \text{ and } \|\theta_{e}(\cdot, s,\tau)\|_{L^\infty} & \le & \frac{C \ln s}{s}\label{estima-thetia---e}
\end{eqnarray}
for $s$ large enough.\\
Now, 
by definition \eqref{c4defini-R--} of the negative component, 
we write
\begin{eqnarray*}
& & \left(\int_\sigma^s  \mathcal{K}(s,\tau) \left( B(q) + R \right)(\tau)  d\tau \right)_- \\
&= &  P_-\left(\chi(\cdot,s)\int_\sigma^s  \mathcal{K}(s,\tau) \left( B(q) + R \right)(\tau)  d\tau \right) =P_-\left( \int_\sigma^s \chi(s) \mathcal{K}(s,\tau) v(\tau)  d\tau \right)\\  
& = &  P_-\left( \int_\sigma^s  \left[ \sum_{j=0}^2 \theta_j(s,\tau) h_j +\theta_-(s,\tau)  \right]  d\tau \right)  
= P_-\left(\int_{\sigma}^s \theta_-( s,\tau) d\tau \right),
\end{eqnarray*}
since 
$$  P_-\left( \int_\sigma^s  \left[ \sum_{j=0}^2 \theta_j(s,\tau) h_j  \right]  d\tau \right) =0 $$
by definition \eqref{c4defini-R--} of the projector $P_-$.\\
Let us recall the following obvious estimate : 
 if $|g(y)| \le m (1 + |y|^3 ) $, then,  we have 
\begin{eqnarray*}
  |P_-(g)(y) | \le Cm (1 + |y|^3).
\end{eqnarray*}
Using   \eqref{estima-thetia---e}, we obtain
\begin{eqnarray*}
\left|\int_{\sigma}^s \theta_-( s,\tau)   d\tau \right| \le \frac{C(s-\sigma) \ln s(1 +|y|^3)}{s^\frac{5}{2}},
\end{eqnarray*}
therefore,
\begin{eqnarray}\label{Duong-}
\left| \left(\int_\sigma^s  \mathcal{K}(s,\tau) \left( B(q) + R \right)(\tau)  d\tau \right)_-  \right|  \le \frac{C(s-\sigma) \ln s}{s^\frac{5}{2}}(1 + |y|^3).
\end{eqnarray}
Similarly, we have  
\begin{eqnarray*}
 & & \hspace{-0.4cm} \left( \int_\sigma^s  \mathcal{K}(s,\tau) \left( B(q) + R \right)(\tau)  d\tau\right)_e =  (1-\chi(s))\int_\sigma^s  \mathcal{K}(s,\tau) \left( B(q) + R \right)(\tau)  d\tau  \\
 & = & \int_{\sigma}^s (1 -\chi(s)) \mathcal{K}(s,\tau) v(\tau)d\tau =  \int_\sigma^s \theta_e(s,\tau) d\tau,
\end{eqnarray*}
and from \eqref{estima-thetia---e},  we obtain
\begin{eqnarray}\label{Duonge}
 \left\|\left( \int_\sigma^s  \mathcal{K}(s,\tau) \left( B(q) + R \right)(\tau)  d\tau \right)_e \right\|_{L^\infty} \le \frac{C(s-\sigma) \ln s}{s}.
\end{eqnarray}

\bigskip

\textbf{Step 3: Conclusion of the proof of Proposition \ref{propo-ODE-mode-V-A-s}}

Since the solution $q(s)$ appears in \eqref{integral-equation} as a sum of 2 terms handled in Step 1 and Step 2, one has simply to add the appropriate estimates from those steps (namely \eqref{Ky-}, \eqref{Kye}, \eqref{Kys0}, \eqref{Duong-} and \eqref{Duonge}) in order to get the desired estimates and conclude the proof of Proposition \ref{propo-ODE-mode-V-A-s}.


\end{proof}


\medskip

\textbf{Part 2: Proof of Proposition \ref{pro-reduction- to-finit-dimensional}}

We claim 
that
Proposition \ref{pro-reduction- to-finit-dimensional} is a consequence of the following:
\begin{proposition}[Improvement on the non-expanding directions]\label{lemma-control-u-P-1}
There exists $K_{4} \ge 1  $ and $A_4 \ge 1$ such that for all  $K\geq K_4 $ and $A \ge A_4$,   there exists  $s_4(K, A) \ge 1$ such that  for all $  s_0  \ge s_4 $,
the following property holds:   Assume that 
 \begin{itemize}
 \item[(a)]
 Initial data $\psi_{d_0,d_1}$ defined 
 by
 \eqref{defi-initial-data},  for some $(d_0,d_1) \in \mathcal{D}_{s_0}$ given in  Proposition \ref{propo-initial-data}.
\item[(b)] 
$q(s) \in V_{K, A}(s)$,    for all $s \in [s_0,
s^*
]$, for some $
s^*
\ge s_0$. 
\end{itemize}  
Then, we have the following estimates for all $s \in [s_0,
s^*
]$:
\begin{equation}\label{improve}
|q_2(s)| \leq \frac{A^2 \ln s}{ s^2} - s^{-3}, \;\;
\left\|\frac{q_-(.,s)}{1+|y|^3}\right\|_{L^\infty} \leq  \frac{A^6\ln s}{2 s^\frac{5}{2}} \text{ and } \|  q_e(.,s)  \|_{L^\infty(\R^N)} \leq \frac{A^7\ln s}{2s}.
\end{equation}
 \end{proposition}
 Let us first use this proposition to derive Proposition \ref{pro-reduction- to-finit-dimensional}, then we will proof Proposition \ref{lemma-control-u-P-1}.

\begin{proof}[Proof of Proposition \ref{pro-reduction- to-finit-dimensional} assuming that Proposition \ref{lemma-control-u-P-1} holds]\label{prood-of-propo-3-6}
Let us take $K=K_4$, $A=\max(A_4, 2 \bar C)$, where $K_4$ and $A_4$ appear in Proposition \ref{lemma-control-u-P-1} and $\bar C$ in item (i) of Proposition \ref{propo-ODE-mode-V-A-s}, $s_0=s_4(K,A)$, and assume hypotheses (a) and (b) stated in Proposition \ref{pro-reduction- to-finit-dimensional}. Clearly, the hypotheses of Proposition \ref{lemma-control-u-P-1} are satisfied.
Thus, estimate \eqref{improve} holds.

\medskip

- \textit{ Proof of item $(i)$:} Since $q(s_*)\in \partial V_{K,A}(s_*)$ by hypothesis (b), using the Definition \ref{shrinking-set} of $V_{K,A}(s_*)$, we see from  \eqref{improve} that only $q_0(s_*)$ and $q_1(s_*)$ may touch the edges of the interval $[-\frac A{s_*^2}, \frac A{s_*^2}]$. By definition \eqref{c4defini-hat-mathcal-V-A} of $\hat V_A(s_*)$, the conclusion follows. 
%


\medskip

- \textit{ Proof of item $(ii)$:} 
Since
$(q_0(s_*),q_1(s_*)) \in \partial  \mathcal{\hat V}_A(s_*)$
by item \textit{(i)}, it follows that 
either  $\left|q_0(s_*) \right| =  \frac{A}{s_*^2} $  or  $ \left|q_1(s_*) \right| = \frac{A}{s_*^2}$. Without loss of generality, we may assume that $ \left|q_0(s_*) \right| =  \frac{A}{s_*^2} $ (the other
case follows similarly).
Hence, 
there exists $ \sigma_0 = 1 $ or $-1$ such that $ q_0(s_*) = \sigma_0 \frac{A}{s_*^2}$. From item $(i)$ in Lemma \ref{estima-priori-q-2},
recalling that $A\ge 2\bar C$,
we see that
$$ \sigma_0 \frac{d}{ds} \left( q_0(s) - \sigma_0 \frac{A}{s^2} \right) \left. \right|_{s = s_*} >0,$$
which shows that the flow of $q_0$ is transverse outgoing on the curve of $\sigma_0\frac A{s^2}$ at $s=s_*$. This implies in fact that $q_0(s)$ will indeed cross that curve, which concludes the proof of 
\eqref{traverse-outgoing crossing}. 
This concludes also the proof of Proposition \ref{pro-reduction- to-finit-dimensional}, assuming that Proposition \ref{lemma-control-u-P-1} holds.
\end{proof}
 
 It remains now to prove Proposition \ref{lemma-control-u-P-1} in order to finish the proof of Proposition \ref{pro-reduction- to-finit-dimensional}.
 
 \medskip
 
\begin{proof}[Proof of Proposition \ref{lemma-control-u-P-1}]
 The proof is quite the same as    
 the proof of
 \cite[Proposition 3.7]{MZdm97},
 thanks to
 Proposition
 \ref{propo-ODE-mode-V-A-s}.  For instance, the estimate for $q_2$ 
 follows similarly as in
 that work, and the proof techniques for $q_-$ and $q_e$ are also the same. For these reasons, 
we will only give  
 the proof  
 for
 $q_-$. 
 Let us consider $K\ge K_3$, $A\ge A_3$, $\ell^*=\ln A$ and $s_0\ge s_3(A,K,\ell^*)$, where the constants $K_3$, $A_3$ and $s_3$ are defined in Proposition \ref{propo-ODE-mode-V-A-s}.
Then, let us assume that hypotheses (a) and (b) stated in Proposition \ref{lemma-control-u-P-1} hold, for some $
s_*
\ge s_0$.

Consider then $s\in [s_0,s_*]$ and let us prove the estimate on $q_-$ stated in \eqref{improve}.  
Clearly,
we are allowed to apply Proposition \ref{propo-ODE-mode-V-A-s}.
 Two cases then arise:

 \medskip
 
 + If $ s - s_0 \le 
 \ln A
 $, then, we 
 clearly can 
 apply  estimate
 \eqref{estima-q-1-y-3-P-1-s-0} in  Proposition
 \ref{propo-ODE-mode-V-A-s},   with 
$ \ell = s- s_0   $  and $\sigma =s_0$, 
and derive
that
 \begin{eqnarray*}
 \left\|\frac{q_-(.,s)}{1+|y|^3}\right\|_{L^\infty}  \le \frac{C \ln s}{s^\frac{5}{2}} (1 + 
 \ln A)
 \le \frac{A^6\ln s}{2 s^\frac{5}{2}}
 \end{eqnarray*}
 provided that $A \ge A_{4,1}$, for some $  A_{4,1}\ge 1$. 

 + If $ s - s_0 \ge 
 \ln A
 $, then,  we can apply estimate
 \eqref{estima-q-1-y-3-P-1-stric} with $\sigma = s-
 \ln A
 $ and $ \ell  =
 \ln A
 $ (=$\ell^*$)
 and derive
 that
 \begin{eqnarray*}
 \left\|\frac{q_-(\cdot,s)}{1+|y|^3}\right\|_{L^\infty} \le C
 \frac{\ln s}{s^{5/2}}
 \left(   A^6 e^{-\frac{\ln A}{2}} + A^7 e^{-(\ln A)^2} + A^2 e^{\ln A}  (1+(\ln A)^2) 
 +\ln A\right)
 \leq  \frac{A^6\ln s}{2 s^\frac{5}{2}},
 \end{eqnarray*}
 provided that $A \ge A_{4,2}$, for some $  A_{4,2}\ge 1$. 
 \medskip
  Finally,  
  fixing $A_4 = \max(A_3, A_{4,1}, A_{4,2})$ then taking
$A \ge A_4$, the result follows and  conclude the proof of Proposition \ref{lemma-control-u-P-1}. Since we have already proved that Proposition \ref{pro-reduction- to-finit-dimensional} follows from Proposition \ref{lemma-control-u-P-1}, this concludes the proof of Proposition \ref{pro-reduction- to-finit-dimensional} too.
\end{proof}





%
%
%

\appendix
\section{Some technical estimates}

In this section, we recall from previous literature 
some
useful and straightforward 
estimates 
on
$V$, $B(q)$ and $R$
defined in \eqref{defi-V}, \eqref{defi-B-q} and \eqref{defi-R}.
This is our statement:
\begin{lemma}\label{lemma-V-q}
For any $K\ge 1$ and $A\ge 1$, there exists $s_5(K,A)$ such that for any $s\ge s_5(K,A)$, if 
$q(s) \in V_{K, A}(s)$, then, the following hold:
\begin{itemize}
\item[$(i)$]   Potential   $V$ defined 
in \eqref{defi-V}:
  \begin{eqnarray*}
\left|   V  (y,s)     +  \frac{1}{4s}   \left( |y|^2 - 2 N \right) \right| \leq  \frac{ C(K)(1 + |y|^4)}{s^2} , \mbox{ whenever }|y| \leq  2 K  \sqrt s,
  \end{eqnarray*}
  and 
  $$ |V(y,s)| \le C, \forall y \in \R^N.$$
  \item[$(ii)$]   Nonlinear term  $B(q)$ defined 
  in
  \eqref{defi-B-q}:
  $$    |B(q)| \le C(K)|q|^2,   \forall  |y| \le 2K \sqrt{s}, $$
  and
  $$   |B(q)|  \le C |q|^{\bar p} , \bar p = \min(2,p).   $$
  \item[$(iii)$]  
Remainder
 term  $R$ defined 
  in  \eqref{defi-R}:
  \begin{eqnarray*}
  \left| R(y,s) -  \frac{a_0}{s^2}   \right|  & \le & \frac{C(K)(1 + |y|^4)}{s^3}, \forall  |y| \le 2K \sqrt{s},  \\
  \| R(.,s)\|_{L^\infty}  & \le &  \frac{C}{s},
  \end{eqnarray*}
  for some $a_0\in \R$.
\end{itemize}
\end{lemma}
\begin{proof}
The proof can be found in  \cite[Lemma B.5]{ZAAihn98},  and \cite[Lemma A.9 for $\alpha =0$]{DNZtunisian-2017}.
\end{proof}







\bibliographystyle{alpha}
\bibliography{mybib}

\end{document}